\pdfoutput=1 
\documentclass[12pt]{extarticle}
\usepackage[utf8]{inputenc}
\usepackage{stmaryrd}
\usepackage{comment}
\usepackage{caption}
\usepackage{multicol}
\usepackage{tikz}
\usepackage{xfrac} 
\oddsidemargin \evensidemargin
\usepackage{floatrow}

\usepackage{amsmath, amssymb, amsfonts}
\usepackage{mathtools}
\usepackage{enumerate}
\usepackage{multicol}
\usepackage[all]{xy}
\usepackage[utf8]{inputenc}
\usepackage{bbm}
\usepackage{mathrsfs}
\usepackage{tikz,pgfplots,tikz-3dplot}

\renewenvironment{thebibliography}[1]{
  \begin{oldthebibliography}{#1}
    \setlength{\itemsep}{0.05em}
    \setlength{\parskip}{0em}
}
{
  \end{oldthebibliography}
}

\usepackage[english]{babel}

\definecolor{MidnightBlue}{rgb}{0.0, 0.33, 0.71}
\definecolor{Blue}{rgb}{0.1, 0.1, 0.44}
\definecolor{teal}{rgb}{0.0, 0.5, 0.5}

\usepackage[colorlinks=true, allcolors=blue,linkcolor=MidnightBlue, citecolor=Blue,backref=page,urlcolor=black]{hyperref}

\definecolor{mycolor1}{rgb}{0.00000,0.44700,0.74100}
\definecolor{mycolor2}{rgb}{0.8500, 0.3250, 0.0980}
\definecolor{mycolor3}{rgb}{0.9290, 0.6940, 0.1250}
\definecolor{mycolor4}{rgb}{0.4940, 0.1840, 0.5560}
\definecolor{mycolor5}{rgb}{0.4660, 0.6740, 0.1880}

\usepackage{amsthm}
\usepackage{cleveref}

\newtheorem{theorem}{Theorem}[section]

\newtheorem{lemma}[theorem]{Lemma}
 
\newtheorem{proposition}[theorem]{Proposition}

\theoremstyle{definition}

\newtheorem{definition}[theorem]{Definition}
\newtheorem{algorithm}[theorem]{Algorithm}

\newtheorem{examplex}[theorem]{Example}
\newenvironment{example}
{\pushQED{\qed}\examplex}
{\popQED\endexamplex}

\newtheorem{remarkx}[theorem]{Remark}
\newenvironment{remark}
{\pushQED{\qed}\remarkx}
{\popQED\endremarkx}

\numberwithin{equation}{section}

\newtheoremstyle{citing}% name
{}%     Space above, empty = `usual value'
{}%     Space below
{\itshape}% Body font
{}%     Indent amount (empty = no indent, \parindent = para indent)
{\bfseries}% Thm head font
{\textbf{.}}% Punctuation after thm head
{.5em}% Space after thm head: " " = normal interword space;
{\thmnote{#3}}% Thm head spec
{\theoremstyle{citing}
}

\DeclareMathOperator{\Hilb}{Hilb}

\DeclareMathOperator{\NN}{\mathbb{N}}
\DeclareMathOperator{\CC}{\mathbb{C}}
\DeclareMathOperator{\RR}{\mathbb{R}}
\DeclareMathOperator{\ZZ}{\mathbb{Z}}
\DeclareMathOperator{\QQ}{\mathbb{Q}}
\DeclareMathOperator{\init}{in}
\DeclareMathOperator{\ind}{ind}

\DeclareMathOperator{\rank}{rank}

\DeclareMathOperator{\GL}{GL}

\DeclareMathOperator{\Sol}{Sol}

\newcommand{\cO}{\mathcal{O}}
\renewcommand{\d}{\mathrm{d}}

\newcommand{\cM}{\mathcal{M}}

\newcommand{\cI}{\mathcal{I}}

\pgfplotsset{compat=1.18}

\title{Border Bases in the Rational Weyl Algebra}
\author{Carlos Rodriguez${}^\flat$ and Anna-Laura Sattelberger${}^\flat$}
\date{}

\begin{document}
\maketitle
\thispagestyle{empty}

\begin{abstract}
Border bases are a generalization of Gröbner bases for {zero-dimensional ideals in} polynomial rings.
In this article, we introduce border bases for a non-commutative ring of linear differential operators, namely the rational Weyl algebra. We elaborate on their properties and present algorithms to compute with them. We apply this theory to represent integrable connections as cyclic $D$-modules explicitly. As an application, we visit differential equations behind a string, a Feynman as well as a cosmological integral. We also address the classification of particular $D$-ideals of a fixed holonomic rank, 
namely the case of linear PDEs with constant coefficients as well as Frobenius ideals. Our approach rests on the theory of Hilbert schemes of points in affine space. 
\end{abstract}

\vspace*{-6mm}
\begin{small}
{\hypersetup{linkcolor=black}
\setcounter{tocdepth}{2}
\tableofcontents
\renewcommand{\baselinestretch}{1.0}\normalsize
}
\end{small}

\vfill

{\small
\noindent ${}^{\flat}$  Max Planck Institute for Mathematics in the Sciences, Inselstra{\ss}e~22, 04103 Leipzig, Germany\\ 
\hspace*{1.4mm} {\tt $\{$carlos.rodriguez and anna-laura.sattelberger$\}$ @ mis.mpg.de}}

\newpage
\addcontentsline{toc}{section}{Introduction}
\section*{Introduction}
Linear partial differential operators with polynomial coefficients are encoded as elements of the Weyl algebra, denoted $D_n=\CC[x_1,\ldots,x_n]\langle \partial_1,\ldots,\partial_n\rangle$. The mathematical field of algebraic analysis investigates systems of linear PDEs by studying left modules over~$D_n$ in algebro-geometric terms.
In this article, we focus on left ideals in the rational Weyl algebra, 
\begin{align*}%\label{eq:Weylrat}
    R_n \,=\, \CC(x_1,\ldots,x_n) \langle \partial_1,\ldots,\partial_n \rangle ,
\end{align*}
which encodes linear PDEs with coefficients in the field of rational functions in the variables \mbox{$x=(x_1,\ldots,x_n)$}.
For computations with these ideals, we resort to the theory of border bases. Border bases are a generalization of Gröbner bases for {ideals of points in} polynomial rings and are numerically better behaved, see~\cite{Stetter}. In more geometric terms, border bases naturally arise in the classification of zero-dimensional ideals $I$ in polynomial rings. There, border bases can be used to represent open subschemes that cover the Hilbert scheme $\Hilb_n^m$ of $m$ points in affine $n$-space, see~\cite{combalg}, and they are characterized in terms of the commutativity of the companion matrices which encode the endomorphisms on $\CC[X_1,\ldots,X_n]/I$ given by the multiplication by the variables~$X_i$, see~\cite{KKR05}. 

In this article, we introduce border bases for left ideals $J\subset R_n$ {of finite holonomic rank}, discuss their properties, and present algorithms to compute with them. Motivated by the characterization of border bases in terms of the companion matrices, we characterize border bases in the rational Weyl algebra in terms of the connection matrices of the ideal, which encode  the multiplication by $\partial_i$ on $R_n/J$. We show in \Cref{prop:borderbasescomm_Rn} that border bases in the rational Weyl algebra---in contrast to the commutative case---are characterized by imposing integrability conditions. 
%This is due to the fact that the endomorphism on~$R_n/J$ given by left multiplication by $\partial_i$ is only $\CC$-linear, but not $\CC(x_1,\ldots,x_n)$-linear. 

As an application of the theory, we show how to represent an integrable connection as a cyclic $D$-module by explicitly constructing generators of a $D$-ideal $I$ such that $D/I$ gives rise to the connection matrices one started with. {That this is possible at all, is explained by the fact that integrable connections are holonomic $D$-modules and these, in turn, are cyclic, see~\cite{HTT08,Stafford78}.}
Each $D_n$-ideal $I$ of holonomic rank~$m$, when expressed as a Pfaffian system, gives rise to an $n$-tuple $(A_1,\ldots,A_n)$ of $m\times m$ matrices. % with entries in $\CC(x_1,\ldots,x_n)$. 
To be precise, we rewrite $I$ as a first-order matrix system
\begin{align*}
\partial_i \bullet F \,=\, A_i \cdot F, \quad i=1,\ldots,n,
\end{align*}
where $F$ is a vector of functions---for instance a vector of master integrals in the study of Feynman integrals---and the $A_i$'s are $m\times m$ matrices with entries in $\CC(x_1,\ldots,x_n)$.
We refer to these matrices as ``connection matrices.''  By construction, the connection matrices fulfill the integrability conditions, i.e.,
%and denote the tuples of such matrices by $\mathcal{A}_n$; it collects tuples of matrices that might occur as connection matrices of a $D_n$-ideal of holonomic rank~$m$. It is the subset
%The set of $n$-tuples of such matrices is the subset
%\begin{align*}%\begin{split}%\label{eq:An}
 %\cA_n  & \, \coloneqq \,  \left\{ (A_1,\ldots,A_n) \, | \, [A_i,A_j] \,=\, \partial_i\bullet A_j-\partial_j\bullet A_i \text{ for all }i \neq j  \right\}
 %\\ & \, \subset \, \mathbb{M}_{m\times m}^n\left(\CC(x_1,\ldots,x_n)\right)
%\end{split}
%\end{align*}
%of the variety of $n$-tuples of $m\times m$ matrices, 
\begin{align*}
    [A_i,A_j] \,=\, \partial_i\bullet A_j-\partial_j\bullet A_i \quad \text{ for all }i \neq j ,
\end{align*}
where entry-wise differentiation of the matrices is meant.
One can translate the problem to $m\times m$ matrices $A$ of rational differential one-forms by passing to the connection form $\d -A\wedge \ $ {while remembering the chosen basis}, where $\d$ denotes the total differential and $A=A_1\d x_1+\cdots + A_n\d x_n$. {We demonstrate this explicitly in \Cref{ex:connPfaff}.}
One sometimes is also interested in the space~$\{A \, | \, \d A=0\}$ of {\em closed} connection matrices (in the sense of closed differential forms), such as for instance connection matrices with logarithmic differentials as entries. If the connection matrix $A$ is closed, then the $A_i$'s are pairwise commuting. One hence finds oneself within varieties of commuting matrices. %This setup is inherently tied to Quot schemes, which are in set-theoretic bijection with stable ADHM data, see~\cite{ADHM, QuotInt}.  
Closed connection matrices also occur in the setup of dimensional regularization of Feynman integrals, in which an additional small parameter~$\varepsilon$ is present: the so-called ``$\varepsilon$-factorized'' form of~\cite{henn2013multiloop} of the connection matrices implies that they are closed.

One can pass from a holonomic $D_n$-ideal{---also called ``maximally over-determined systems'' in the algebraic analysis literature}---to an integrable connection, for instance by using the package {\tt ConnectionMatrices}~\cite{ConnMatM2} in the computer algebra software {\em Macaulay2}~\cite{M2}. The package builds on Gröbner basis computations in the rational Weyl algebra, see~\cite{SST00}. In the present article, we tackle the reverse direction: we utilize the theory of border bases to associate an $R_n$-ideal to an integrable connection, expressing the connection as a cyclic module explicitly. Deriving $D$-ideals from a matrix system of PDEs is also of current interest in the study of Feynman integrals, see e.g.~\cite{Didealbanana,diffspaceFeynman}. $D$-modules and Feynman integrals have had a very fruitful overlap. For example, the theory of holonomic $D$-modules was used to prove a conjecture in physics about the finiteness of master integrals---in mathematical terms, that the holonomic rank of the annihilating $D$-ideal of a Feynman integral is finite, see \cite{KK77,SmirnovNumberFeynman}.

As a further application, we visit a cosmological correlation function, and show that in this example, border bases are well-suited to make certain features like symmetry manifest---in a more immediate way than Gröbner bases would.
We also address the classification of certain $D$-ideals of a fixed holonomic rank~$m$ by resorting to the theory of Hilbert schemes of points in affine space from the commutative setup. The ideals which fit that setup are linear PDEs with constant coefficients as well as Frobenius ideals, i.e., $D_n$-ideals that can be generated by polynomials in the Euler operators $\theta_i=x_i\partial_i$, $i=1,\ldots,n$.

\bigskip
\noindent{\bf Notation and conventions.} The letter $\NN=\ZZ_{\geq 0}$ denotes the non-negative integers. The ($n$-th) Weyl algebra is denoted $D_n=\CC[x_1,\ldots,x_n]\langle \partial_1,\ldots,\partial_n\rangle$, or just $D$, if the number of variables is clear from the context. We denote the rational Weyl algebra by $R_n=\CC(x_1,\ldots,x_n)\langle \partial_1,\ldots,\partial_n\rangle$. We denote the action of a differential operator on a function by a $\bullet$, e.g., $\partial_i\bullet f=\partial f/\partial x_i$.
When speaking about $D$-ideals and $D$-modules, we always mean left $D$-ideals and left $D$-modules, and likewise for the rational Weyl algebra. If the number of variables is clear from the context, we denote by $\CC(x)$ the field of rational functions $\CC(x_1,\ldots,x_n)$ in the variables $x=(x_1,\ldots,x_n)$. We typically denote $D_n$-ideals by the letter~$I$ and $R_n$-ideals by~$J$. {We here take coefficients in $\CC$; in applications, one typically encounters $\QQ$ or $\RR$ instead. For polynomial rings $S=\CC[X_1,\ldots,X_n]$, we use $X_i$ as variable names for a better distinction from the non-commutative case. For the sake of better readability, we avoid the $[\cdot]$-notation for equivalence classes; e.g., $\partial_i\in R_n/J$ stands for~$[\partial_i]=\partial_i+J\in R_n/J$.}

\bigskip
\noindent{\bf Outline.} 
\Cref{sec:background} recalls some background. This includes the topics of $D$-ideals and integrable connections, Hilbert schemes of points  as well as border bases in commutative polynomial rings. In \Cref{sec:borderWeyl}, we introduce border bases for {ideals of finite holonomic rank in} the rational Weyl algebra. We discuss theoretical properties of border bases and present algorithms to compute with them. As an application, we explain how to construct a $D$-ideal that represents an integrable connection as a cyclic $D$-module explicitly and also address moduli problems of certain \mbox{$D$-ideals} of a fixed holonomic rank by resorting to the theory of Hilbert schemes of points. \Cref{sec:explappl} presents applications in physics. We visit differential equations behind a string integral in genus zero, generic-mass sunrise  integrals in dimensional regularization, and the correlation function of the cosmological two-site graph. We also comment on $\varepsilon$-factorized connection matrices.

\section{Background}\label{sec:background}
\subsection{Writing systems of linear PDEs in matrix form}
We first recall some background from the theory of $D$-modules and integrable connections, see for instance \cite{HTT08,SST00}.
We denote by $D_n=\CC[x_1,\ldots,x_n]\langle \partial_1,\ldots,\partial_n\rangle$ the ($n$-th) Weyl algebra. It is obtained from the free $\CC$-algebra generated by the variables $x_1,\ldots,x_n$ and partial differential operators $\partial_1,\ldots,\partial_n$, by imposing the following relations: all generators are assumed to commute, except $x_i$ and $\partial_i$. They fulfill the Leibniz rule, i.e., their commutator obeys $[\partial_i,x_i]=1$, $i=1,\ldots,n$.
Similarly, we denote by $R_n=\CC(x_1,\ldots,x_n)\langle \partial_1,\ldots,\partial_n\rangle$ the ($n$-th) rational Weyl algebra, with corresponding commutator relations. The {\em holonomic rank} of a $D_n$-ideal~$I$, denoted $\rank(I)$, is the dimension of the underlying $\CC(x)$-vector space of the $R_n$-module $R_n/R_nI$, i.e.,
\begin{align}
\rank(I) \, \coloneqq \, \dim_{\CC(x)}(R_n/R_nI),
\end{align} 
where $\CC(x)$ abbreviates $\CC(x_1,\ldots,x_n)$. In the same way, for ideals $J \subset R_n$, we refer to the $\CC(x)$-dimension of $R_n/J$ as the holonomic rank of the $R_n$-ideal $J$. If $I$ is holonomic, then on any simply connected domain in $\CC^n$ outside the singular locus of~$I$, the holonomic rank equals the dimension of the $\CC$-vector space of holomorphic solutions to the system of PDEs encoded by~$I$. This follows from the theorem of Cauchy--Kovalevskaya--Kashiwara.

Let $I$ be a $D$-ideal of holonomic rank $m$, and let $(s_1=1,s_2,\ldots,s_m)$ be a $\CC(x)$-basis of~$R_n/R_nI$. The $s_i$'s can be chosen to be monomials in the $\partial_i$'s, and w.l.o.g.\ we can assume $s_1=1$. For a solution  $f\in \Sol(I)$ to~$I$, denote $F=(f,s_2\bullet f,\ldots,s_m\bullet f)^\top$. Then there are unique $m\times m$ matrices $A_1,\ldots,A_n$ with entries in $\CC(x)$ that fulfill
\begin{align}\label{eq:connmatrices}
    \partial_i \bullet F \,=\, A_i\cdot F, \qquad i=1,\ldots,n
\end{align}
for \underline{any} $f\in \Sol(I)$. This system is called {\em Pfaffian system} of~$I$ (with respect to the chosen basis), see~\cite[p.~38]{SST00}. We refer to the matrices in \eqref{eq:connmatrices} as the {\em connection matrices} of~$I$.  The entries of the connection matrices can be obtained by a Gröbner basis reduction of the $\partial_is_j$ modulo~$I$. The left-hand side of \eqref{eq:connmatrices} is the vector $(\partial_i \bullet f,(\partial_i\cdot s_2)\bullet f,\ldots,(\partial_i\cdot s_m)\bullet f)^\top$, where $\cdot$ is the multiplication of differential operators.
The connection matrices of a $D_n$-ideal $I$ only depend on the $R_n$-ideal $R_nI$. {In what follows, we therefore work with ideals $J=R_nI\subset R_n$ in the rational Weyl algebra mostly.} In particular, the  connection matrices could equally be read from the Weyl closure~\cite{WeylClosureTsai} of~$I$, namely the $D_n$-ideal $W(I)=R_nI\cap D_n$, since $R_nI=R_nW(I)$ as $R_n$-ideals.

\begin{remark}
In the study of scattering amplitudes and cosmology, it is common to obtain systems of linear PDEs of the form in~\eqref{eq:connmatrices} for a vector of holonomic functions $F=(F_1,F_2,\ldots,F_m)^\top$. Examples of these functions include string amplitudes, cosmological correlators, and Feynman loop integrals. Physicists refer to the integrals $F_1,\ldots,F_m$ as {\em master integrals}. 
We can relate this notion to $\CC(x)$-bases of a cyclic $R_n$-module (i.e., a module of the form~$R_n/R_nI$) that consist of monomials in the $\partial_i$'s. The $R_n$-ideal $J=R_nI$ is derived from a $\CC(x)$-combination of the integrals $F_1,\ldots,F_m$, which we demonstrate explicitly at the example of a sunrise Feynman integral in \Cref{sec:sunrise}, see~\Cref{ex:gaugesunrise}. This is one of the examples we present in which one of the master integrals suffices to derive the $R_n$-ideal~$J$. 
\end{remark}

By construction, the connection matrices fulfill the integrability conditions, i.e., 
\begin{align}
\label{eq:integrabilityConditions}
  [A_i,A_j] \,=\,  \partial_i \bullet A_j -\partial_j\bullet A_i \, ,
\end{align}
for all $i,j=1\ldots,n$,
where entry-wise differentiation is meant.
If $\widetilde{F}=g\cdot F$ for some invertible matrix $g\in \operatorname{GL}_m(\CC(x))$, the transformed system then reads as $\partial_i \bullet \widetilde{F}=\widetilde{A}_i\cdot\widetilde{F}$ for
\begin{align}\label{eq:gauge_trans}
    \widetilde{A}_i \, =\, gA_ig^{-1}+\frac{\partial g}{\partial x_i}g^{-1} \, .
\end{align}
The matrix $\widetilde{A}_i$~\eqref{eq:gauge_trans} is the {\em gauge transform} of $A_i$ with respect to the gauge matrix~$g$. Compared to a similarity transform for changes of basis, an additional term is required, namely the second summand on the right hand side of~\eqref{eq:gauge_trans}.

We now turn to a more geometric perspective. Let $X$ be a smooth algebraic variety. Endowing an $\mathcal{O}_X$-module $\mathcal{M}$ with the structure of a $\mathcal{D}_X$-module is equivalent to giving a $\CC_X$-linear map $\nabla$ that fulfills the Leibniz rule, called a {\em connection} on $\cM$,
\begin{align*}
    \nabla \colon \, \cM \longrightarrow \cM \otimes_{\mathcal{O}_X} \Omega_X^1
\end{align*}
that is flat (also called {\em integrable}), i.e., $\nabla^2=0$. For $X=\mathbb{A}_{\CC}^n$ the affine $n$-space, modules over $D_n$ correspond to sheaves of $\mathcal{D}_{X}$-modules that are quasi-coherent over~$\mathcal{O}_{X}$ (see \cite[Proposition 1.4.4]{HTT08}).  For $\cM$ locally free, choosing a (local) identification $\cM\cong \mathcal{O}_X^m$, one can write $\nabla=\d -A\wedge \,$ with $A$ an $m\times m$ matrix of differential one-forms. The integrability conditions translate as $\d A - A\wedge A=0$. Rewriting $\nabla\colon \Theta_X \to \mathcal{E}\!\operatorname{nd}_{\mathbb{C}_X}(\cM)$ ($\mathcal{O}_X$-linear) by the tensor-hom adjunction and passing to the stalk at the generic point,~$(0)$, to arrive at the field of rational functions, the $\nabla(\partial_i)$ recover the matrices $A_i$ from~\eqref{eq:connmatrices}. The other way round, $A=A_1\d x_1+\cdots+A_n\d x_n$. To be precise, there is a dualization swept under the rug: the connection matrices as defined here actually describe the $D$-module structure on the $D$-module that is dual to~$D/I$, also see~\cite[Section~2.1]{ConnMatM2} for a brief discussion.

\smallskip
In the following example, we demonstrate how to pass from an integrable connection to a Pfaffian system, while carrying along a chosen basis.
\begin{example}[$n=2,m=3$]\label{ex:connPfaff}
 Let $V=\CC(x_1,x_2)^3 $ %\cong (R_2/\langle \partial_1,\partial_2))^3$
 together with the trivial connection, i.e., $\nabla = \d -A\wedge \,$ with $A$ being the $3\times 3$ zero matrix in the basis of standard unit vectors $(e_1,e_2,e_3)$. One has $(V,\nabla)\cong (R_2/\langle \partial_1,\partial_2\rangle)^3$ and the flat sections of $(V,\d)$ are (component-wise) constant bivariate functions to $\CC^3$. This determines a holonomic (and hence, in particular, cyclic) \mbox{$R_2$-module} of rank~$3$. We now construct an associated Pfaffian system. The vector $v=(x_1^2,x_1,1)$ can be used as a cyclic vector in the sense that the left $R_2$-linear morphism determined by
\begin{align}\label{eq:iso}
   \varphi \colon \,  R_2/J \stackrel{\cong}{\longrightarrow}  \CC(x_1,x_2)^3, \quad 1\mapsto v
\end{align}
is an isomorphism of $R_2$-modules,
where $J=\langle \partial_1^3,\partial_2\rangle $ is a holonomic ideal of rank $3$, namely the kernel of $R_2\rightarrow \CC(x_1,x_2)^3, \ 1 \mapsto v$. 
In \eqref{eq:iso}, the $R_2$-module structure on $R_2/J$ is the one induced by $R_2$, and the $R_2$-module structure on $\CC(x_1,x_2)^2$ is given by the trivial connection, i.e., usual differentiation of rational functions (component-wise).
To determine a $\CC(x_1,x_2)$-basis $(s_1,s_2,s_3)$ of $R_2/J$ in which we write the Pfaffian system of $J$, we now determine preimages of the standard unit vectors under~$\varphi$. 
Since $\frac{1}{2}\partial_1^2\bullet v = e_1$, we set $s_1=\partial_1^2$. We set \begin{align*}
s_2\coloneqq \partial_1\cdot \big(1-\frac{x_1^2}{2}\partial_1^2\big) \quad \text{and} \quad  s_3\coloneqq 1-\frac{x_1^2}{2}\partial_1^2-x_1\cdot s_2,
\end{align*}
resulting in $s_2\bullet v = e_2$ and $s_3\bullet v=e_3$.
In the $\CC(x_1,x_2)$-basis $(s_1,s_2,s_3)$ of $R_2/J$, the Pfaffian system $(A_1,A_2)$ indeed consists of the zero matrices only, since
\begin{align}\label{ex:trconn}
\partial_i \bullet 
\begin{pmatrix}
s_1\bullet f\\ s_2 \bullet f\\ s_3\bullet f
\end{pmatrix} 
\,=\, 
    \partial_i \bullet  \begin{pmatrix}
        \partial_1^2\bullet f \\
        -\frac{x_1^2}{2}  \partial_1^3\bullet f -x_1  \partial_1^2\bullet f+  \partial_1\bullet f \\
        \frac{x_1^3}{2}\partial_1^3\bullet f+ \frac{x_1^2}{2}\partial_1^2\bullet f-x_1\partial_1\bullet f +f
    \end{pmatrix}
\,=\, 
\begin{pmatrix} 0 \\ 0 \\ 0 \end{pmatrix}, \quad i\,=\, 1,2,
\end{align}
for any $f\in \Sol(J)=\CC \cdot \{1,x_1,x_1^2\}$.
Also any set of standard monomials of the $R_2$-ideal could serve as $\CC(x_1,x_2)$-basis. The matrices in the Pfaffian system~\eqref{ex:trconn} would, however, not be the zero matrices anymore---they get transformed according to the gauge transform.
\end{example}

In practice, starting from a $D_n$-ideal $I$ of finite holonomic rank, its connection matrices can be systematically computed with the help of Gröbner bases in the rational Weyl algebra. This is implemented in the package {\tt ConnectionMatrices}~\cite{ConnMatM2} in the open-source computer algebra software {\em Macaulay2}~\cite{M2}. An alternative, fast implementation of Pfaffian systems, building on so-called ``Macaulay matrices,'' is provided in~\cite{MacaulayFeynman}. 

\begin{example}[{\cite[Section 3.3]{Britto:2021prf}}]\label{ex:stringy}
For $m=3$, a family of string integrals is given by the~vector
\begin{align}
\vec{F}(x_1,x_2) \,= \int_0^{x_1} x_0^{s_{12}}(1-x_0)^{s_{25}}\left[\prod_{j=1}^{2}(x_{j}-x_0)^{s_{2j}} \right] \left(\begin{smallmatrix}
  \frac{s_{12}}{x_0} \\
  \frac{s_{12}}{x_0}+\frac{s_{23}}{x_0-x_1} \\
  \frac{s_{12}}{x_0}+\frac{s_{23}}{x_0-x_1}+\frac{s_{24}}{x_0-x_2} 
\end{smallmatrix}\right)  \d x_0 \,  \eqqcolon \begin{pmatrix}
    F_1\\F_2\\F_3
\end{pmatrix},
\end{align}
where the $s_{ij}\in\mathbb{C}\,\backslash\mathbb{Z}$ are fixed complex numbers. These integrals are closely related to scattering amplitudes of open strings.  The connection matrices for this  vector of string integrals in the variables $(x_1,x_2)$ are:
\begin{align}\begin{split}\label{eq:A3A4stringy}
    A_1 &\,=\, \left(\!\begin{array}{ccc}
    \frac{s_{12}+s_{23}}{x_1}&-\frac{s_{12}}{x_1}&0\\
    -\frac{s_{24}}{x_1-x_2}-\frac{s_{25}}{x_1-1}&
    \frac{s_{24}+s_{23}}{x_1-x_2}+\frac{s_{25}}{x_1-1}&
    -\frac{s_{23}}{x_1-x_2}+\frac{s_{23}}{x_1-1}\\
    -\frac{s_{25}}{x_1-1}&\frac{s_{25}}{x_1-1}&\frac{s_{23}}{x_1-1}
    \end{array}\!\right),
    \\
    A_2 &\,=\, \left(\!\begin{array}{ccc}
    \frac{s_{24}}{x_2}&\frac{s_{12}}{x_2}&-\frac{s_{12}}{x_2}\\
    \frac{s_{24}}{x_2}-\frac{s_{24}}{x_2-x_1}&\frac{s_{12}}{x_2}+\frac{s_{24}+s_{23}}{x_2-x_1}&-\frac{s_{12}}{x_2}-\frac{s_{23}}{x_2-x_1}\\
    0&-\frac{s_{25}}{x_2-1}&\frac{s_{24}+s_{25}}{x_2-1}
    \end{array}\!\right) .
 \end{split} 
 \end{align}
This pair of matrices commutes and we can write
\begin{align*}
    \d P(x_1,x_2) = A_1 \, \d x_1 + A_2 \, \d x_2, 
\end{align*}
with the entries of $P(x_1,x_2)$ being logarithms of rational functions. 
%
\begin{comment}
where $P(x_1,x_2)$ is the following matrix, with logarithmic $0$-forms as entries: 
%
{\footnotesize
\begin{align*}%\begin{split}
   % P(x_1,x_2) =& 
   \left(\!\begin{array}{ccc}
    (s_{12}+s_{23}) \log x_1+s_{24} \log x_2 &0& -s_{12} \log \, x_2\\
    s_{24} \log \, x_2 -s_{25} \log \, (x_1-1)&
    s_{12} \log \, x_2+s_{25} \log \, (x_1-1)&
    -s_{12} \log \, x_2 + s_{23} \log \, (x_1-1) \\
    -s_{25} \log \, (x_1-1)&
    0&
    s_{23} \log \, (x_1-1) + (s_{24}+s_{25}) \log \, (x_2-1)
    \end{array}\!\right)  
    \\ 
   % &
    +\left(\!\begin{array}{ccc}
    0&
    s_{12}\log \,\frac{x_2}{x_1}&
    0\\
    -s_{24} \log \, (x_1-x_2)&
    (s_{24}+s_{23}) \log \,(x_1-x_2)&
    -s_{23} \log (x_1-x_2)\\
    0
    &
    s_{25} \log \, \frac{x_1-1}{x_2-1}&
    0
    \end{array}\!\right) \, \, .
   % \end{split}
\end{align*}
}
\end{comment}
\end{example}

\subsection{Hilbert schemes of points}\label{sec:Hilb}
The classification of zero-dimensional ideals in polynomial rings is a classical object of study and gets addressed by the theory of Hilbert schemes of points. 

Let $S=\CC[X_1,\ldots,X_n]$.
The {\em Hilbert scheme of $m$ points in affine $n$-space},
\begin{align}\label{eq:Hilbm}
    \Hilb_n^m \,=\, \left\{ I\subset S \,|\, \dim_{\CC} (S/I) =m \right\} ,
\end{align}
classifies ideals $I\subset S$ whose quotient ring is $m$-dimensional as a $\CC$-vector space. It is known to be connected, which was proven by Hartshorne~\cite{Har66} with the help of {distractions} of ideals. A combinatorial construction of affine subschemes of the Hilbert scheme which cover $\Hilb_n^m$ is provided in~\cite[Chapter 18]{combalg}. 
As we will argue in \Cref{sec:classFrobconst}, this theory can be used for the classification of Frobenius ideals, which are encoded by ideals in the commutative subring $\CC[\theta_1,\ldots,\theta_n]$ of~$D_n$, as well as the case of linear PDEs with constant coefficients, which are encoded by ideals in the polynomial ring $\CC[\partial_1,\ldots,\partial_n]$.
Before explaining the general theory, we present some examples. 

\begin{example} \label{ex:MxMyCommute}
Let \mbox{$S=\CC[X,Y]$}.
In this case, $\Hilb_2^3$ classifies ideals whose variety in the plane is zero-dimensional, containing three points (counted with multiplicity). By \cite[Theorem~18.4]{combalg}, this Hilbert scheme can be covered by three affine charts corresponding to monomial bases $(Y^2,Y,1)$, $(X,Y,1)$, and $(X^2,X,1)$. %, yielding a $6$-dimensional parameter space. 
Each element of $\Hilb^3_2$ in the second chart can be represented as
\begin{align}\label{eq:I3}
    I \,=\, \left\langle X^2-aX-bY-c, \, XY-dX-eY-f, \, Y^2-gX-hy-i \right\rangle  \,\subset \, \CC[X,Y] \, .
\end{align}
Consider the multiplication by $X$ and $Y$ on the quotient $\CC[X,Y]/I$. %$X\cdot,Y\cdot \in \operatorname{End}_{\CC}(\CC[X,Y]/I)$. 
In the $\CC$-basis $(X,Y,1)$ of~$\CC[X,Y]/I$, these $\CC$-linear endomorphisms are represented by the matrices
\begin{align*}
    M_X\,=\, \begin{pmatrix}
       a & d & 1\\
       b & e& 0\\
       c & f & 0
    \end{pmatrix} \quad \text{and} \, \quad  M_Y\,=\, \begin{pmatrix}
       d & g & 0\\
       e & h& 1\\
       f & i & 0
    \end{pmatrix}.
\end{align*}
In order to be a Gröbner basis, all S-pairs of the generators of $I$~\eqref{eq:I3} need to reduce to zero. Writing this out, one observes that this is equivalent to requiring that the matrices $M_X$ and $M_Y$ commute. This occurs if and only if the following three equations are satisfied:
\begin{align*}
f = b g -d e \, ,\quad c = -a e+b d-b h+e^2 \, , \quad i = -a g+d^2-d h+e g   \, .  
\end{align*}
Thus, we get a $6$-dimensional affine chart for $\Hilb^3_2$ from the ideal above. This affine chart corresponds to the monomial basis $(X,Y,1)$ and is denoted by $U_{2+1}$ in \cite[Section 18.1]{combalg}.
\end{example}

\begin{example}[Based on {\cite[Example~18.6]{combalg}}]\label{ex:bordernongroebner} 
Consider the ideal $I\subset S=\CC[X,Y]$ generated by the four polynomials
\begin{align}\begin{split}
\label{ex:BorderBasis}
%I=\left\langle 
p_1 \,=\, X^2-X Y-2 X+Y+1, \quad p_2 \,=\,X^2 Y-2 X Y+2
   Y-1, \\ p_3 \,=\, Y^2-XY+X-1, \quad p_4 \,=\, X Y^2-X
   Y+X+Y-2
   %\right \rangle 
   \, .
\end{split}\end{align}
We first make a comment about Gröbner bases. Note that there is no term order that picks out the first terms of the polynomials in \eqref{ex:BorderBasis} as initial terms: $X\prec Y$ would imply that $X^2 \prec XY$ and $Y \prec X$ that $Y^2 \prec XY$. If we instead chose an order that picks out $XY$ as the initial term of $p_3$, the $S$-pair of $p_3$ and $p_4$ would not reduce to zero. Buchberger's $S$-pair criterion would now imply that $\{p_1,p_2,p_3,p_4\}$ is not a Gröbner basis of~$I$.

The $\CC$-vector space underlying the quotient ring $S/I$ is $4$-dimensional. In the $\CC$-basis $(1,X,Y,XY)$ of~$S/I$, the matrices representing the multiplication by $X$ and $Y$ are
\begin{align*}
    M_X \,=\, \begin{pmatrix}
        0 & -1 & 0 & 1 \\
        1 & 2 & 0 & 0 \\
        0 & -1 & 0 & -2 \\
        0 & 1 & 1 & 2 \\
    \end{pmatrix} \quad \text{and} \quad  \, M_Y \,=\, \begin{pmatrix}
       0 & 0 & 1 & 2 \\
       0 & 0 & -1 & -1 \\
       1 & 0 & 0 & -1 \\
       0 & 1 & 1 & 1 \\
    \end{pmatrix},
\end{align*}
which can be conveniently read from the the generators in \eqref{ex:BorderBasis} of $I$ due to their special form. One can check that the matrices $M_X$ and $M_Y$ commute. As we explain in a later part of this article, this implies that $p_1,\ldots,p_4$ as in~\eqref{ex:BorderBasis} constitute a \textit{border basis} of $I$ with respect to the order ideal $\mathcal{O}=\{1,X,Y,XY\}$.\footnote{Border bases are distinct from Gröbner bases: one can easily read off the multiplication matrices from the border basis, but it doesn't need to satisfy Buchberger's criterion. If it is ``connected to $1$,'' the criterion it has to satisfy is that the multiplication matrices commute, see~\cite{mourrain1999new}.} The matrices $M_X$ and $M_Y$ then are called \emph{formal multiplication matrices}. In terms of the Hilbert scheme of four points in the plane, $\Hilb_2^4$, this ideal naturally lives in the affine chart corresponding to the monomial basis $(1,X,Y,XY)$, which is denoted by $U_{2+2}$ in {\cite[Section 18.2]{combalg}} the theory of which we recall in \Cref{sec:borderbases}.

A Gröbner basis for $I$ with respect to degree reverse lex with $X\succ Y$ is given by: 
\begin{align}
 I \,=\, \left\langle \underline{X Y}-X-Y^2+1, \, \underline{X^2}-3
   X-Y^2+Y+2, \, X+\underline{Y^3}-2\right\rangle \, .
   \label{eq:idealInDifferentCharts}
\end{align}
For this monomial ordering, the standard monomials of $I$ are $\{1,X,Y,Y^2\}$,  which corresponds to another affine chart of the Hilbert scheme $\Hilb_2^4$,  denoted $U_{2+1+1}$ in {\cite[Section~18.1]{combalg}}. In this different basis for $S/I$, we represent the multiplication by $X$ and $Y$ by matrices $\widetilde{M}_X$ and~$\widetilde{M}_Y$: 
\begin{align}
    \widetilde{M}_X\,=\, \begin{pmatrix}
         0 & -2 & -1 & 1 \\
         1 & 3 & 1 & 0 \\
        0 & -1 & 0 & -1 \\
         0 & 1 & 1 & 1 
    \end{pmatrix} \quad \text{and} \quad \, \widetilde{M}_Y\,=\, \begin{pmatrix}
        0 & -1 & 0 & 2 \\
        0 & 1 & 0 & -1 \\
       1 & 0 & 0 & 0 \\
       0 & 1 & 1 & 0 
    \end{pmatrix}.
\end{align}
%The different bases for $S/I$ are related by a matrix $g\in \GL_4(\mathbb{C})$ :
The change of basis from $\{ 1,X,Y,XY\}$ to $\{1,X,Y,Y^2\}$ is encoded by the invertible matrix
\begin{comment}
\begin{align*}
    \begin{pmatrix}
    1&
    X&
    Y &
    X Y
    \end{pmatrix}^\top
=
g \cdot \begin{pmatrix}
    1&
    X &
    Y &
    Y^2
 \end{pmatrix}^\top \, ,
\end{align*}
with $g$ given by
\end{comment}
\begin{align*}
B \,=\, \begin{pmatrix}
   1 & 0 & 0 & -1 \\
 0 & 1 & 0 & 1 \\
 0 & 0 & 1 & 0 \\
 0 & 0 & 0 & 1 
 \end{pmatrix} \, \in \, \GL_4(\mathbb{C}).
\end{align*} 
The matrices $ \widetilde{M}_X$ and $M_X$ (respectively, $ \Tilde{M}_Y$ and $M_Y$) are related via a conjugation by~$B$:
\begin{align*}
\widetilde{M}_i \,=\, B \, M_i \, B^{-1} , \quad  i=X,Y \, .
\end{align*}
encoding a usual change of basis.  
\end{example}
The previous example showed a change of basis for the vector space $S/I$ with a nice geometric meaning: the ideal in \eqref{eq:idealInDifferentCharts} belongs to two different affine charts of the Hilbert scheme $\Hilb_2^4$. We will proceed to explain these affine charts of the Hilbert scheme of points from the point of view of border bases. 

\medskip
The Hilbert scheme $\Hilb_n^m$ of $m$ points in $\mathbb{A}_{\CC}^n$ classifies all zero-dimensional ideals \mbox{$I\subset S=\CC[X_1,\ldots,X_n]$} that give rise to $\CC$-vector spaces $S/I$ of dimension~$m$. These ideals then can be grouped according to which $m$ monomials give a basis for them.

An \textit{order ideal}\footnote{The term ``ideal'' is to be understood in the setup of partially ordered sets, not in the  algebraic~sense.} is a non-empty subset $\lambda \subset \mathbb{N}^n$ such that
\begin{align*}
\mathbf{u} \in \lambda \, ,\mathbf{v} \in \mathbb{N}^n \, ,\mathbf{v} \leq \mathbf{u} \implies \mathbf{v}\in \lambda \, ,
\end{align*}
where $(u_1,u_2,\ldots,u_n)\leq(v_1,v_2,\ldots,v_n) \iff u_1 \leq v_1,u_2<v_2,\ldots,u_n\leq v_n$ coordinate-wise. Equivalently, an order ideal is the set of exponents on monomials outside of a monomial ideal, see \cite[Section 18.4]{combalg}. The Hilbert scheme $\Hilb_n^m$ is covered by open subschemes,  
%(affine schemes when $S=\CC[x,y]$), 
with each open subscheme $U_\lambda\subset \Hilb_n^m$ consisting of all ideals $I \in \Hilb_n^m$ such that $\{\mathbf{X}^\mathbf{u} | \mathbf{u\in \lambda}\}$ is a $\CC$-basis of $S/I$, with $|\lambda|=m$. The equations defining the affine subscheme $U_\lambda$ of $\Hilb_n^m$ can be conveniently spelled out with the use of \textit{border bases}, whose definition we recall in \Cref{sec:borderbases}. These equations are obtained from requiring that certain matrices commute, see \Cref{prop:Ulambda_equations}. We have already seen an example of this in \Cref{ex:MxMyCommute}.

\subsection{Border bases in polynomial rings}\label{sec:borderbases}
This section recalls border bases in polynomial rings and closely follows the presentation~\cite{KKR05} of Kehrein, Kreuzer, and Robbiano. For additional references, see for instance~\cite{KR05,KR11}.

Let $S=\CC[X_1,X_2,\ldots,X_n]$ be the polynomial ring in $n$ variables and denote by $\mathbb{T}_n$ the set of monomials in $X_1,\ldots,X_n$. Let $\lambda\subset \NN^n$ be an order ideal. Sometimes, one also refers to the set of monomials $\mathcal{O}_\lambda=\{\mathbf{X}^\mathbf{u}|\mathbf{u}\in \lambda\}\subset \mathbb{T}_n$ as an {\em order ideal}. The aim of border bases is to explicitly construct generators of an ideal $I_\lambda\subset S$ such that $S/I_\lambda$ is an $m$-dimensional vector space over~$\CC$ with a basis given by the monomials $\mathcal{O}_\lambda=\{\mathbf{X}^\mathbf{u}|\mathbf{u}\in \lambda\}$. Before giving the definition of a border basis, we recall required concepts.

\begin{definition}\label{def:corners}
Let $\mathcal{O}_{\lambda}$ be an order ideal. The minimal generators of the monomial ideal generated by $\mathbb{T}_n\setminus \mathcal{O}_{\lambda}$ are called the {\em corners} of $\mathcal{O}_{\lambda}$.
\end{definition}

\begin{definition}\label{def:border}
The \textit{border} of the order ideal $\mathcal{O}_\lambda$ is 
\begin{align}\label{eq:border}
\partial \mathcal{O}_\lambda \,=\, (X_1 \mathcal{O}_\lambda \cup X_2 \mathcal{O}_\lambda\cup \cdots  \cup X_n \mathcal{O}_\lambda  ) \backslash \mathcal{O}_\lambda \, .
\end{align}
\end{definition}
The {\em first border closure} of $\cO_\lambda$ is $\overline{\partial \cO}\coloneqq \cO\cup \partial \cO$.
Iteratively, for $k\geq 2$, one defines the {\em $k$-th border} as 
\begin{align*}
    \partial^k\cO_{\lambda} \, \coloneqq \, \partial\left( \overline{\partial^{k-1}\cO_{\lambda}}\right) \,.
\end{align*}
and the {\em $k$-th border closure} of $\cO_\lambda$ as $\overline{\partial^k\cO_\lambda}\coloneqq \overline{\partial^{k-1}\cO_\lambda}\cup \partial^k \cO_\lambda$. %It is an order ideal again. 
For convenience, one sets $\partial^0 \cO_\lambda=\overline{\partial^0\cO_\lambda}=\cO_\lambda$.
The {\em index} of a monomial $t\in \mathbb{T}_n$, denoted $\ind_{\lambda}t$, is the smallest $k$ such that $t \in \partial^k\mathcal{O}_{\lambda}$. The {\em index} of a polynomial $f\in S$ is the maximal index of its monomials, i.e., for $f = \sum_\mathbf{a} c_\mathbf{a} \mathbf{X}^\mathbf{a}$, its index is
\begin{align}\label{eq:index}
    \ind_{\lambda}(f) \,=\, \max \{\ind_{\lambda} \mathbf{X}^\mathbf{a}  \, |\, c_\mathbf{a} \neq 0\}
\end{align}
We sometimes drop the subscript from $\cO_{\lambda}$, since the $\lambda$ can be understood implicitly from the elements of $\cO=\cO_{\lambda}$.
We will illustrate these concepts with an example. 
\begin{example}
\label{ex:BorderIdeal1}
Consider the order ideal $\mathcal{O}=\{1,X_1,X_2,X_1 X_2\}\subset \mathbb{T}^2$. We can also identity this order ideal by the exponents of its monomials, i.e., the set $\lambda = \{(0,0),(1,0),(0,1),(1,1)\}\subset \NN^2$. The first and second borders  of this order ideal are given by $\partial \mathcal{O}=\{X_1^2,X_1^2 X_2,X_1 X_2^2,X_2^2\}$ and $\partial^2 \mathcal{O}=\{X_1^3,X_1^3 X_2,X_1^2 X_2^2, X_2^3 X_1,X_2^3\}$.
We visualize these integer points for $\mathcal{O}$ and its borders in the plane, see \Cref{fig:border}.
\begin{figure}[h]
\begin{tikzpicture}[scale=.82, transform shape]
\draw[->] (0,0)--(4,0) node[right]{\large $i$};
\draw[->] (0,0)--(0,4) node[above]{\large $j$};
\filldraw (0,0) circle (3.8pt);
\filldraw (1,0) circle (3.8pt);
\filldraw (1,1) circle (3.8pt);
\filldraw (0,1) circle (3.8pt);
\draw[thick,teal] (0,2) circle (3.65pt);
\draw[thick] (1,2) circle (3.65pt);
\draw[thick,teal] (2,0) circle (3.65pt);
\draw[thick] (2,1) circle (3.65pt);
\draw (-.3,-.3) node {$0$};
\draw (1,-.35) node {$1$};
\draw (-.3,1) node {$1$};
\draw[thick] (0,3) node {\large $\times$};
\draw[thick] (1,3) node {\large $\times$};
\draw[thick] (2,2) node {\large $\times$};
\draw[thick] (3,0) node {\large $\times$};
\draw[thick] (3,1) node {\large $\times$};
\end{tikzpicture}
\caption{The black dots at $(i,j)$  correspond to the elements $X_1^i X_2^j$ of the order ideal $\mathcal{O}$ of \Cref{ex:BorderIdeal1}. We also picture the elements in the border $\partial \mathcal{O}$ with circles (with the turquoise circles corresponding to the corners) and the elements of the second border $\partial^2 \mathcal{O}$ with crosses.
}
\label{fig:border}
\end{figure}
\end{example}

A {\em border prebasis} with respect to an order ideal $\lambda$ is a set of polynomials each of which is a $\CC$-linear combination of one term in $\partial \mathcal{O}_\lambda$ and all the terms of~$\mathcal{O}_\lambda$ in the following sense. If $\mathcal{O}_\lambda = \{t_1,t_2,\ldots,t_m\}\subset \mathbb{T}_n$ is an order ideal and $\partial \mathcal{O}_\lambda = \{b_1,b_2,\ldots,b_p\}$ its border, then an $\cO_{\lambda}$-border prebasis is a set of polynomials
\begin{align*}
G_\lambda \,=\, \{g_1,g_2,\ldots,g_p\} \, ,
\end{align*}
where 
\begin{align}\label{eq:pbc}
g_j \,=\, b_j - \sum_{i=1}^{m} c_{i,j}t_i  
\end{align}
with $c_{i,j}\in \CC$. The element $g_j$ as in~\eqref{eq:pbc} is said to be {\em marked} by the border element~$b_j$.

An important fact about border prebases is that a division algorithm exists which allows us to compute normal forms, in analogy to Gröbner basis theory.

\begin{proposition}[{\cite[Proposition~4.2.10]{KKR05}}]
Given an $\mathcal{O}_\lambda$-border prebasis $\{g_1,\ldots,g_p\}$ with $\mathcal{O}_\lambda = \{t_1,t_2,\ldots,t_m\}$ and border $\partial \mathcal{O}_\lambda = \{b_1,b_2,\ldots,b_p\}$, then for any $f\in S$ we can find polynomials $f_j\in S$ and coefficients $c_j\in \CC$ such that
\begin{align}\label{eq:div_alg}
f \,=\, f_1 g_1 +\cdots+f_p g_p + c_1 t_1 + \cdots +c_m t_m 
\end{align}
and $\deg(f_i) \leq \ind_{\lambda}(f)-1$ for all $i$ with $f_i g_i\neq 0$.
\end{proposition}
Such $f_i$'s and $c_j$'s can be obtained as follows.

\begin{algorithm}[Border division algorithm]\label{alg:borderdivision}
\begin{enumerate}[D1]
\item[]
    \item Set $f_1=\cdots=f_p=0$, $c_1=\cdots=c_m=0$, and $h=f$.
    \item If $h=0$, return $(f_1,\ldots,f_p,c_1,\ldots,c_m)$ and stop.
    \item If $\ind_\lambda(h)=0$, there exist $c_1,\ldots,c_m \in \CC$ such that $h=c_1t_1+\cdots+c_mt_m$. Find such~$c_j$'s, return $(f_1,\ldots,f_p,c_1,\ldots,c_m)$, and stop.
    \item If $\ind_\lambda(h)>0$, let $h=a_1h_1+\cdots+a_sh_s$ such that $a_i\in \CC \backslash \{0\}$ and $h_i\in \mathbb{T}_n$ satisfy $\ind_\lambda(h_1)=\ind_\lambda(h)$. Find the minimal $i$ s.t.\ $h_1=t'b_i$ with $t'\in \mathbb{T}_n$ of degree $\ind_{\lambda}(h)-1$. Subtract $a_1t'g_i$ from~$h$, add $a_1t'$ to $f_i$,  and continue with step D2.
\end{enumerate}
\end{algorithm}
The algorithm terminates after a finite number of iterations, because there are only finitely many terms of index less or equal to the index of~$h$. The representation~\eqref{eq:div_alg} does not depend on the choice of $h_1$ in step $D4$, because $h$ gets replaced by terms of strictly smaller index.
However, the constants $\{c_1,\ldots,c_m\}$ obtained this way are not unique using a border prebasis, and we will need to impose additional conditions to achieve uniqueness. 

Before addressing that, we associate matrices $M_i$, $i=1,\ldots,n$, to a border prebasis.
The $k$-th column of $M_i$ is defined to be
\begin{align}\label{eq:formalmult}
    \left({M_i}\right)_{\ast k} \,=\, 
    \begin{cases}
      e_r & \text{if} \ X_it_k=t_r, \\
  (c_{1,s},\ldots,c_{m,s})^\top & \text{if} \ X_it_k=b_s \, ,
    \end{cases}
\end{align}
where $e_r$ denotes the $r$-th standard unit vector in $\RR^m$ and the $c_{j,s}$'s are as in~\eqref{eq:pbc}. We note that the $k$-th column of $M_i$ consists precisely of $c_1,\ldots,c_m$ as in~\eqref{eq:div_alg} obtained from applying the division algorithm to~$X_it_k$, i.e., the elements of the order ideal multiplied by~$X_i$. Because of that, the $M_i$'s in~\eqref{eq:formalmult} are called {\em formal multiplication matrices}. 

\medskip
We now recall the definition of border bases. 
\begin{definition}
Let $I\subset S=\CC[X_1,\ldots,X_n]$ be a zero-dimensional ideal, $\mathcal{O}_\lambda = \{t_1,t_2,\ldots,t_m\}\subset \mathbb{T}_n$ an order ideal, and $G=\{g_1,g_2,\ldots,g_p\}\subset I$ an $\mathcal{O}_\lambda$-border prebasis. The set $G$ is an {\em $\mathcal{O}_\lambda$-border basis} of $I$ if the residue classes of $t_1,t_2,\ldots,t_m$ form a $\CC$-vector space basis of~$S/I$. 
\end{definition}
In particular, this also means that any border basis of $I$ also generates $I$ as an $S$-ideal. 

In~\cite{KK06}, algorithms are provided to compute a border basis from a given set of generators, relating to a degree-compatible term ordering. As also pointed out therein, border bases behave numerically better than Gröbner bases, see also~\cite{Stetter}. In~\cite{Kaspar13}, the original border basis algorithm is extended in a way that the resulting border basis does not relate to any term ordering; instead of initial terms, one uses ``markings.''
Concerning software, border bases can be computed with the package {\tt borderbasix}~\cite{borderbasix}, which can also be used to solve zero-dimensional systems.
It is important to note that not every order ideal gives rise to a border basis. \Cref{ex:symmnoborder} demonstrates that fact in the rational Weyl algebra. For a given set of distinct points in~$\CC^n$, \cite{HKP19} provides an algorithm to compute all order ideals $\cO$ for which there exists an $\cO$-border basis of their ideal in $\CC[X_1,\ldots,X_n]$.

\smallskip

A crucial theorem in the characterization of border bases in terms of formal multiplication matrices is the following.

\begin{theorem}[{\cite[Theorem 4.3.17]{KKR05}}]\label{thm:borderbasescomm}
Let $\mathcal{O}_\lambda = \{t_1,t_2,\ldots,t_m\}$ be an order ideal. Then an $\cO_{\lambda}$-border prebasis $G=\{g_1,\ldots,g_p\}$ is an $\cO_{\lambda}$-border basis of $I=\langle g_1,\ldots,g_p\rangle$ if and only if the formal multiplication matrices are pairwise commuting. 
\end{theorem}
These relations of pairwise commutation encode precisely the equations that describe the affine subschemes $U_\lambda$ of the Hilbert scheme $\Hilb_n^m$ from \Cref{sec:Hilb}. 

We summarize some results from \cite[Chapter~18]{combalg} in 
\begin{proposition}\label{prop:Ulambda_equations}
The Hilbert scheme $\Hilb_n^m$ is covered by affine open subschemes $U_\lambda$, where $U_\lambda = \{I\subset S \,|\,  \cO_\lambda \textrm{ is a $\mathbb{C}$-basis of }S/I  \}$. Thus, every ideal $I_\lambda\in U_\lambda$ admits an $\cO_\lambda$-border basis. Moreover, we can give affine coordinates to $U_\lambda$ in the following way: Let $I_\lambda$ be generated by an $\cO_\lambda$-border prebasis. Then the condition that the formal multiplication matrices computed from this $\cO_\lambda$-border prebasis commute cut out an affine subscheme $U_\lambda$~of~$\Hilb_n^m$. 
\end{proposition}
More explicitly, let $\cO_\lambda=\{t_1,t_2,\ldots,t_m\}$, $\partial \cO_\lambda=\{b_1,b_2,\ldots,b_p\}$, and 
\begin{align*}
I_\lambda \,=\, \big\langle b_1 - \sum_{j=1}^{m}c_{1,j}\,t_j, \, b_2 - \sum_{j=1}^{m}c_{2,j}\,t_j,\ldots, \, b_p - \sum_{j=1}^{m}c_{p,j}\,t_j  \big\rangle.
\end{align*}
The entries of the formal multiplication matrices computed from this $\cO$-border prebasis then are either $0$, $1$, or some $c_{i,j}$, $i\in\{1,2,\ldots,p\}$, $j\in\{1,2,\ldots,m\}$. Thus, the conditions for the formal multiplication matrices to commute generate an ideal  $J_{\textrm{comm}}\subset \mathbb{C}[c_{i,j}] $. Then the affine coordinates for $U_\lambda$ are the equivalence classes $[c_{i,j}] \in \mathbb{C}[c_{i,j}]/J_{\textrm{comm}}$. We saw %a concrete 
an example of this in \Cref{ex:MxMyCommute}, where ${a,b,d,e,g,h}$ served as coordinates for $U_{2+1} \in \Hilb_2^3$.

\medskip
Another feature of border bases is that they can be used to determine reduced Gröbner bases. 
Let $I\subset S$ be a zero-dimensional ideal and $\prec$ a term order on $\mathbb{T}_n$. We denote by $\mathcal{O}_{\prec}(I)$ the order ideal consisting of the standard monomials of a Gröbner basis of $I$ w.r.t.~$\prec$.

\begin{proposition}[{\cite[Proposition~4.3.6]{KKR05}}]\label{prop:borderreducedGr}
 Let $I$ and $\prec$ be as above. Then there exists a unique $\mathcal{O}_\prec(I)$-border basis $G$ of $I$, and the reduced Gröbner basis of $I$ with respect to~$\prec$ is the subset of $G$ consisting of the polynomials that are marked by the corners of~$\cO_\prec(I)$.
\end{proposition}

\section{Border bases in the rational Weyl algebra}\label{sec:borderWeyl}
In this section, we transfer and adapt the theory of border bases to the non-commutative setup of algebras of differential operators. Among other applications, we use them to derive a cyclic $D$-module from given connection matrices.

\subsection{Definition and properties}
We denote by $\mathbb{D}_n$ the multiplicative monoid of monomials in the variables $\partial_1,\ldots,\partial_n$. Instead of $\CC$, we now consider $\CC(x)=\CC(x_1,\ldots,x_n)$ as our field of coefficients. The concepts of order ideals and border prebases from \Cref{sec:borderbases} then generalize straightforwardly to the rational Weyl algebra, and so do border bases. We show that many of the properties of border bases in the commutative case have an analog in the rational Weyl algebra.

\begin{definition}
%Let $I\subset D_n$ be a $D_n$-ideal and $\mathcal{O}_\lambda=\{t_1,t_2,\ldots,t_m\}\subset \mathbb{D}_n$ be an order ideal. An $\mathcal{O}_\lambda$-border prebasis $\{g_1,\ldots,g_p\} \subset R_nI$ is a {\em border basis} of $R_n I$ if the residue classes of $t_1,\ldots,t_m$ form a basis of the $\mathbb{C}(x)$-vector space basis underlying the $R_n$-module~$R_n/R_n I$. 
Let $J\subset R_n$ be an $R_n$-ideal and $\mathcal{O}_\lambda=\{t_1,t_2,\ldots,t_m\}\subset \mathbb{D}_n$ be an order ideal. An $\mathcal{O}_\lambda$-border prebasis $\{g_1,\ldots,g_p\} \subset J$ is a {\em border basis} of $J$ if the residue classes of $t_1,\ldots,t_m$ form a basis of the $\mathbb{C}(x)$-vector space basis underlying the $R_n$-module~$R_n/J$. 
\end{definition}

\begin{theorem}\label{thm:uniquenessborder}
Let $\mathcal{O}_\lambda=\{t_1,t_2,\ldots,t_m\}$ be an order ideal, let $J$ be an ideal of $R_n$ of holonomic rank~$m$, and assume that the residue classes of the elements in $\mathcal{O}_\lambda$ form a  $\mathbb{C}(x)$-vector space basis of~$R_n/J$. Then
\begin{enumerate}
    \item There exists a unique $\mathcal{O}_\lambda$-border basis $G$ of~$J$.
    \item Let $G'$ be an $\mathcal{O}_\lambda$-border prebasis whose elements are in~$J$. Then $G'$ is the $\mathcal{O}_\lambda$-border basis of~$J$.
\end{enumerate}
\end{theorem}
\noindent This theorem is the rational Weyl algebra version of the first two statements of \cite[Theorem~4.3.4]{KKR05}. 
\begin{proof}
To prove Claim~1, we let $\partial \mathcal{O}_\lambda=\{b_1,\ldots,b_p\}$. For every $i\in\{1,\ldots,p\}$, the residue class of $b_i$ in $R_n/J$ is linearly dependent of the residue classes of the elements of~$\mathcal{O_\lambda}$. Therefore, for each $i$, $J$ contains a polynomial of the form $g_i = b_i -\sum_{j=1}^m \alpha_{ij}t_j$ such that $\alpha_{ij}\in \mathbb{C}(x)$. Then $G=\{g_1,\ldots,g_p\}$ is an $\mathcal{O}_\lambda$-border prebasis, and hence an $\mathcal{O}_\lambda$-border basis of~$J$. Let $G'=\{g_1',g_2',\ldots,g_p'\}$ be another $\mathcal{O}_\lambda$-border basis of $J$, i.e., for $i\in\{1,\ldots,p\}$ we have that $g_i'=b_i - \sum_{j=1}^m\alpha_{ij}'t_j$. If, for contradiction, there is a $b_j$ for which $g_j$ and $g_j'$ differ, then the difference $g_j-g_j'\in J$ encodes a non-trivial linear relation among the equivalence classes of the elements of $\cO_{\lambda}$. This contradicts the hypothesis of the residue classes of the elements of~$\mathcal{O}_\lambda$ being a $\CC(x)$-basis of~$R_n/J$. Thus, Claim~1 is proven.

Claim~2 follows from Claim 1, as $G'$ in Claim~2 is actually an $\mathcal{O}_\lambda$-border basis.
\end{proof}

The generalization of \cite[Proposition~4.3.2]{KKR05} to the non-commutative setup reads as follows, with the same proof.
\begin{proposition}
Any border basis of $J\subset R_n$ in particular generates the ideal~$J$.
\end{proposition}

\begin{remark}
Given an ideal $J \subset R_n$ of holonomic rank~$m$, not every choice of $\mathcal{O}_\lambda$ allows one to construct an $\mathcal{O}_\lambda$-border basis of~$J$. However, one can find an order ideal $\mathcal{O}_\lambda$ for which we can construct an $\mathcal{O}_\lambda$-border basis, for instance by taking $\mathcal{O}_{\lambda}$ to be the set of standard monomials of a Gröbner basis of~$J$ with respect to a term order~$\prec$ on~$\mathbb{D}_n$. %, which we denote by~$\cO_{\prec}(J)$.
\end{remark}

\begin{example}\label{ex:symmnoborder}
    Let $I=\langle \partial_x^2+\partial_y^2-2,\partial_x\partial_y-1\rangle \subset D_2$ which has $\rank(I)=4$. Although the given generators are symmetric in the sense that they are invariant under swapping $x$ and~$y$, no border basis of $J=R_nI$ exists for the order ideal $\mathcal{O}=\{1,\partial_x,\partial_y,\partial_x\partial_y\}$, which is the only symmetric order ideal consisting of four elements.
\end{example}

Border bases can also help to compute reduced Gröbner bases. For that, let $\prec$ be a term order on $\mathbb{D}_n=\{\partial^\alpha\}_{\alpha \in \NN^n}$. Let $P=\sum_{\alpha}r_{\alpha}\partial^{\alpha}\in R_n$, with $r_{\alpha}\in \CC(x).$
The {\em initial term} of $P$, denoted $\init_{\prec}(P)$, is the term $\partial^{\alpha}$ which is the $\prec$-largest monomial occurring in $P$ with non-zero coefficient. For an $R_n$-ideal~$J$, the set $\mathbb{D}_n\setminus \{\init_{\prec}(P) | P\in J\}\subset \NN^n$ is an order ideal which we denote by~$\mathcal{O}_{\prec}(J)$. The {\em corners} of $\cO_{\prec}(J)$ are defined just as in the commutative case, see \Cref{def:corners}. 
\begin{proposition}\label{prop:groebnercorner}
Let $\prec$ be a term order on~$R_n$ and $J$ an $R_n$-ideal. Then there exists a
unique $\mathcal{O}_{\prec}(J)$-border basis $G$ of $J$, and the reduced $\prec$-Gröbner basis of $J$ is the
subset of $G$ consisting of the polynomials that are marked by the corners of~$\mathcal{O}_{\prec}(J)$. 
\end{proposition}
\begin{proof}
By standard arguments from Gröbner bases, see~\cite[p.~33]{SST00}, the residue classes of the elements of $\cO_{\prec}(J)$ are a $\CC(x)$-basis of $R_n/J$. Then, by \Cref{thm:uniquenessborder}, there exists a unique $\cO_{\prec}(J)$-border basis of~$J$. Now let $b \in \mathbb{D}_n\setminus \cO_{\prec}(J)$ be a corner. The element of the minimal $\prec$-Gröbner basis of $J$ with leading term $b$ has the form $b - \operatorname{NF}_{\prec,J}(b)$, with the normal form $\operatorname{NF}_{\prec,J}(b)$ of $b$ being contained in the $\CC(x)$-span of $\mathcal{O}_{\prec}(J)$. Since the $\cO_{\prec}(J)$-border basis of $J$ is unique, this Gröbner basis element coincides with the border basis element marked by~$b$. The reducedness of the resulting Gröbner basis follows from minimality of the corners and the particular shape of the border basis elements.
\end{proof}

The border division algorithm adapted to the rational Weyl algebra reads as follows.
Let~$f\in R_n$ and $\{g_1,\ldots,g_p\}\subset R_n$ be an $\cO_\lambda$-border prebasis of an $R_n$-ideal~$J.$ 
\begin{algorithm}[Border division in the rational Weyl algebra]\label{alg:borderdivWeyl}
\begin{enumerate}[B1]
\item[]
    \item Set $f_1=\cdots=f_p=0$, $c_1=\cdots=c_m=0$ and $h=f$.
    \item If $h=0$, return $(f_1,\ldots,f_p,c_1,\ldots,c_m)$ and stop.
    \item If $\ind_{\lambda}(h)=0$, there exist $c_1,\ldots,c_m \in \CC(x_1,\ldots,x_n)$ such that $h=c_1t_1+\cdots+c_mt_m$. Find such~$c_j$'s, return $(f_1,\ldots,f_p,c_1,\ldots,c_m)$ and stop.
    \item If $\ind_{\lambda}(h)>0$, let $h=a_1h_1+\cdots+a_sh_s$ such that $a_i\in \CC(x) \backslash \{0\}$ and $h_i\in \mathbb{D}_m$ satisfy $\ind_\lambda(h_1)=\ind_\lambda(h)$ and such that $h_1$ is of maximal total degree in the $\partial_i$'s among those.  Find the minimal $i$ such that $h_1=t'b_i$ with $t'\in \mathbb{D}_n$ of degree $\ind_\lambda(h)-1$. Subtract $a_1t'g_i$ from~$h$, add $a_1t'$ to $f_i$, and continue with step~B2.
\end{enumerate}
\end{algorithm}
The algorithm returns a tuple $(f_1,\ldots,f_p,c_1,\ldots,c_m)$
such that
\begin{align}\label{eq:fdivWeyl}
f \,=\, f_1g_1 + \cdots + f_pg_p + c_1t_1 + \cdots + c_mt_m
\end{align}
and $\deg(f_i) \leq \ind_\lambda(f) - 1$ for all $i$ with $f_ig_i\neq0$.
\Cref{alg:borderdivWeyl} is guaranteed to terminate, since the index of $h$ is still guaranteed to decrease after a finite number of iterations. In comparison to \Cref{alg:borderdivision}, due to the Leibniz rule, extra iterations of the fourth step might be required to decrease the index of~$h$. The representation in~\eqref{eq:fdivWeyl} is independent of the choice of~$h_1$ in step B4---due to the Leibniz rule, the index might not decrease after a single iteration, but the degree does. Hence, the reduction of different terms with the same index and total degree do not interfere with~one~another.

\medskip

Let $\cO=\{t_1,\ldots,t_m\}$ be an order ideal and $J\subset R_n$ of holonomic rank $m$ such that  $([t_1],\ldots,[t_m])$ is a $\CC(x)$-basis of~$R_n/J$. The entries of the multiplication matrices~$M_{\partial_i}$ again can be computed by applying the border division algorithm to the $\partial_it_j$'s with respect to the $\cO$-border  basis of $J$. The connection matrices of $J$ as in~\eqref{eq:connmatrices} differ from them by a transpose, i.e., $A_i=M_{\partial_i}^\top$. We give some more details on that in the next remark.

\begin{remark}\label{rem:connmulttranspose}
The left multiplication by $\partial_i$ is only a $\CC$-linear endomorphism on $R_n/J$. It is not $\CC(x)$-linear---instead, it needs to be extended using the Leibniz rule. Hence, the fact that $\partial_i$ and $\partial_j$ commute as differential operators 
does not imply that the multiplication matrices commute. Instead, they fulfill the integrability conditions. Since the connection matrices $A_i$ as in \eqref{eq:connmatrices} differ by a transpose from the matrices $M_{\partial_i}$ as in \eqref{eq:formalmult} obtained from a border basis, the integrability conditions for the multiplication matrices read as 
\begin{align}\label{eq:integrmult}
[M_{\partial_i},M_{\partial_j}] \,=\, \partial_j\bullet M_i-\partial_i\bullet M_j, \quad i,j=1,\ldots,n,
\end{align}
in contrast to $[A_i,A_j]=\partial_i\bullet A_j-\partial_j\bullet A_i$ from \eqref{eq:integrabilityConditions}.
\end{remark}

In the next section, we investigate the formal multiplication matrices in the case of border prebases and explain in which cases they fulfill the integrability conditions.

\subsection{Characterization in terms of integrability}
In the commutative case, \Cref{thm:borderbasescomm} states that a border prebasis is a border basis if and only if the formal multiplication matrices commute. In the rational Weyl algebra, this instead reads as follows.

\begin{theorem}\label{prop:borderbasescomm_Rn}
%Let $I\subset D_n$ be a $D_n$-ideal and $\mathcal{O}_\lambda=\{t_1,t_2,\ldots,t_m\}\subset \mathbb{D}_n$ be an order ideal.
Let $J\subset R_n$ be an $R_n$-ideal and $\mathcal{O}_\lambda=\{t_1,t_2,\ldots,t_m\}\subset \mathbb{D}_n$ be an order ideal.
An $\cO_{\lambda}$-border prebasis $\{g_1,\ldots,g_p\}\subset J$ of $J$ is a border basis of $J$ if and only if the 
formal multiplication matrices $M_{\partial_i}$, $i=1,\ldots,n$, that are read from the $\cO_{\lambda}$-border prebasis, fulfill the integrability conditions~\eqref{eq:integrmult}. 
\end{theorem}
Before proving the theorem, we would like to highlight its consequence for a $D_n$-ideal~$I$: Just from the integrability conditions for the formal multiplication matrices being fulfilled, one can conclude that a border prebasis is in fact a border basis of $R_nI$ and hence that $I$ has holonomic rank~$m$.

\begin{proof}
If $\{g_1,\ldots,g_p\}$ is an $\cO_\lambda$-border basis of~$R_nI$, then $[t_1],\ldots,[t_m]\in R_n/J$ are a \mbox{$\CC(x)$-basis} of $R_n/J$ by definition. Since $M_{\partial_i}=A_i^\top$, see~\Cref{rem:connmulttranspose}, they satisfy the integrability conditions.

For the reverse direction, let $V_m$ denote the $\CC(x)$-vector space spanned by $t_1,\ldots,t_m$. We can assume $t_1=1$. It is a left $R_n$-module with action given by:
\begin{align}
\label{eq:RnModule_def}
\begin{split}
\varphi\colon \,  R_n \times V_m &\rightarrow V_m\, ,  \\\
 \big(\partial_k , \sum_{j=1}^{m} c_j t_j \big)&\mapsto \sum_{j=1}^{m}  (\partial_k\bullet c_j) t_j + \sum_{j=1}^{m}\sum_{l=1}^{m}c_j (A_{k})_{jl} t_l  \, ,
\end{split}
\end{align}
where $c_j \in \mathbb{C}(x)$ and $A_i$ here denotes $M_{\partial_i}^\top$. The action $\varphi$ extends by the following relations for $P,Q \in R_n$, $v\in V$, and~$\lambda \in \CC(x)$:
\begin{align}
\begin{split}
\varphi(P+Q,v)& \,=\, \varphi(P,v)+\varphi(Q,v), \\
\varphi(PQ,v) & \,=\, \varphi(P,\varphi(Q,v)) ,\\
\varphi(\lambda \, P,v)& \,=\, \lambda \cdot \varphi(P,v)\, .
\end{split}
\end{align}
Then for $i,j=1\ldots,n$, the following two properties hold:
\begin{subequations}
\begin{align}
\label{eq:RnModule_first}
([\partial_i ,\partial_j],  \sum_{k=1}^m c_k t_k) & \,\mapsto\,  0 , \\
\label{eq:RnModule_second}
([\partial_i ,x_j]
,\sum_{k=1}^m  c_k t_k
) & \,\mapsto \, \delta_{ij }\sum_{k=1}^m  c_k t_k  \, .
\end{align}
\end{subequations}
Property \eqref{eq:RnModule_first} follows form the assumed integrability condition of the formal multiplication matrices, and Property \eqref{eq:RnModule_second} follows from~\eqref{eq:RnModule_def} by the Leibniz rule. Thus, we have a well-defined left $R_n$-action on~$V_m$.

We are going to prove that $V_m$ is cyclic as a left $R_n$-module. We proceed by induction on the degree of the elements of~$V_m$. We start with the base case $t_1=1$. We have that:
\begin{align*}
\varphi(t_1 , t_1) = 1 \, .
\end{align*}
For the induction step, let $t_i = \partial_j t_k$. Then, we have
\begin{align*}
\varphi(t_i , t_1) = \varphi(\partial_j t_k , t_1) = \varphi(\partial_j,\varphi ( t_k , t_1))=\varphi (\partial_j , \sum_{l=1}^m(A_k)_{1l} t_l)=\varphi(\partial_j , t_k) \nonumber = \sum_{l=1}^m (A_j)_{kl} t_l = t_i\, , \,
\end{align*}
where we have used twice the definition of the formal multiplication matrices. We have thus obtained a surjective left $R_n$-linear map 
\begin{align*}
\widetilde{\Theta}\colon \, R_n \longrightarrow V_m \,  , \quad P\mapsto \varphi(P,t_1).
\end{align*}
We also have an induced isomorphism of left $R_n$-modules 
\begin{align}
\Theta \colon \, R_n /K \longrightarrow V _m \, ,
\end{align}
with $K = \ker \,  \widetilde{\Theta}$. In particular, the residue classes $t_1+K$, $t_2+K$, $\ldots$, $t_m+K$ are linearly independent over $\mathbb{C}(x)$. We then show that $J\subseteq K$. Let $b_j = \partial_k t_l$ be one of the elements of the border, with $g_j = b_j - \sum_{i=1}^{m} \alpha_{ij} t_i$. We then have on $V_m$:
\begin{align}
\begin{split}
\varphi(g_j , t_1) & \,=\,  \varphi(b_j,  t_1) - \sum_{i=1}^m \alpha_{ij} \varphi( t_i , t_1) \,=\, \varphi(\partial_k , t_l) - \sum_{i=1}^m \alpha_{ij} t_i  \\
& \,=\, \sum_{i=1}^m(A_k)_{li} t_i - \sum_{i=1}^m \alpha_{ij }t_i  \,=\, 0 \, .
\end{split}
\end{align}
The last equality follows from the definition of the formal multiplication matrices~\eqref{eq:formalmult}. Therefore, we have that $g_j \in \ker \widetilde{\Theta}$ for $j=1,\ldots,p$, so that $J \subseteq K$.
We thus have a surjective homomorphism 
\begin{align*}
\Psi \colon \, R_n / J \longrightarrow R_n/K \, , \quad [P]_{J} \mapsto [P]_K,
\end{align*}
and since the set $\{t_1+J,\ldots,t_m+J\}$ generates the $\mathbb{C}(x)$-vector space underlying $R_n/J$, and since the set $\{t_1+K,t_2+K,\ldots,t_m+K\}$ is $\mathbb{C}(x)$-linearly independent, both sets must be bases and $J=K$. This shows that $R_n/J$ and $V_m$ are isomorphic as left $R_n$-modules and that $G$ is an $\mathcal{O}$-border basis of~$J$.
\end{proof}

\subsection{From connection matrices to cyclic \texorpdfstring{$D$}{D}-modules}\label{sec:connbord}
From given connection matrices, one can construct a border basis of an associated $R_n$-ideal. Border bases hence enable us to conveniently compute a $D$-ideal that represents an integrable connection as a cyclic $D$-module. We now explain that procedure. 

Let connection matrices $A_1,\ldots,A_n\in \CC(x)^{m\times m}$ be given, together with a choice of basis $(f_1,\ldots,f_m)$ of the $\CC(x)$-vector space $\CC(x)^m$. A key assumption in the following construction is that one of $\{f_1,f_2,\ldots,f_m\}$ is a cyclic vector $v$ of the $R_n$-module $(\CC(x)^m,\nabla=\d -A\wedge \, )$.
Without loss of generality, $v=f_1$. In practice, it might require some effort to construct a cyclic vector from a given basis, see \Cref{ex:connPfaff}. In \Cref{sec:sunrise}, we  provide some heuristics how to identify such an element in applications to Feynman integrals, relying on certain contractions of edges of the underlying Feynman diagram. There, cyclic generators will correspond to maximally uncontracted Feynman diagrams, referred to as ``top sector'' or ``top topology'' in that setup, see \cite[Chapter 6]{WeinzierlFeynmanIntegralsBook}.

We next aim to find an order ideal $\cO$ of the  form $\{ \partial^{I_1},\ldots,\partial^{I_m} \} $ with $I_1,\ldots,I_m\subset \NN^n$, which consists of monomials in the $\partial_i$'s only, such that there exists an invertible matrix $g\in \CC(x)^{m\times m}$ for which the following identity holds:
\begin{align}\label{ex:gauge_cyclic}
    \left(\cO \bullet f_1\right)^\top 
\,=\,
g\cdot \begin{pmatrix}
    f_1 \\
    f_2\\
    \vdots
    \\
    f_m
\end{pmatrix} .
\end{align}
Such an $\cO$ then yields a $\CC(x)$-basis. Note also that the entries of $g$ can be computed from the connection matrices $A_1,A_2,\ldots,A_n$.
We denote the gauge transformed matrices by $\widetilde{A}_i$. Each element $b_i \in \partial \cO$ in the border of~$\cO$ can be written as a multiple of one of the basis elements $\partial^{I_k}$ by some~$\partial_j$. The difference of $b_i$ by the $k$-th entry of $\widetilde{A}_j\cdot (\partial^{I_1},\ldots,\partial^{I_m})^\top$ yields an operator~$P_i$, which is marked by~$b_i$. The integrability of the connection matrices guarantees that $P_i$ is independent of the choice of~$j$, in case there is an ambiguity. We denote by $J_{\cO}$ the $R_n$-ideal generated by the operators $P_i$,  $i=1,\ldots,| \partial \cO|$, arising from the border of $\cO$.

\begin{proposition}
The operators $P_i$, $i=1,\ldots,|\partial \cO|$, as described above, are an $\cO$-border basis of the $R_n$-ideal $J_{\cO}$.
In particular, $J_{\cO}$ has holonomic rank~$m$. 
\end{proposition}
\begin{proof}
The ideal $J_{\cO}$ is generated by a border prebasis. From this border prebasis, we can read off the formal multiplication matrices. 
Because it is obtained from a $\mathbb{C}(x)$-basis change from an integrable connection, they fulfill the integrability conditions, and hence the border prebasis of $J_{\cO}$ is a border basis, via \Cref{prop:borderbasescomm_Rn}. Hence, the holonomic rank is the one expected, namely $\rank (J_{\cO}) =m$.
\end{proof}

\begin{proposition}\label{prop:sameRideal}
Let $J_1,J_2$ be $R_n$-ideals of holonomic rank $m$ for which there exist $P_1,\ldots,P_m\in R_n$ such that both
\begin{enumerate}[(i)]
    \item the equivalence classes of $P_1,
    \ldots,P_m$ are $\CC(x)$-bases of both $R_n/J_1$ and $R_n/J_2$ and
    \item the connection matrices of $J_1$ and $J_2$ in the basis from (i) coincide. 
\end{enumerate}
Then $J_1=J_2$ as $R_n$-ideals.
\end{proposition}
\begin{proof}
Consider the $R_n$-linear morphism given by $R_n / J_1\longrightarrow R_n / J_2,$  $ [1]_{J_1}\mapsto [1]_{J_2}$ using the connection matrices. Due to $(i)$ and $(ii)$, this map is well-defined, and so is $R_n/J_2\longrightarrow R_n/J_1,$  $ [1]_{J_2}\mapsto [1]_{J_1}$, so that we obtain both $J_1 \subseteq J_2$ and $J_2 \subseteq J_1$.
\end{proof}

\begin{lemma}\label{prop:invgauge}
Let $\cO_1$, $|\cO_1|=m$, be an order ideal and a system of integrable matrices in this basis of $\CC(x)^m$ be given. Denote by $J_{1}$ the associated $R_n$-ideal as described above. Let $\cO_2=g\cO_1$ be another order ideal with the change of basis being encoded by $g\in \operatorname{GL}_m(\CC(x))$. Denote by $J_2$ the $R_n$-ideal associated to the gauge transformed matrices obtained from those for $\cO_1$ via a gauge transformation with respect to~$g$. Then $J_{1}=J_{2}$.
\end{lemma}

The lemma is a direct consequence of \Cref{prop:sameRideal}.
We showcase the statement of the lemma with the running example of \cite{ConnMatM2}, which is Example~2.1 therein.
\begin{example}
 Let $I=\langle x\partial_x^2-y\partial_y^2+\partial_x-\partial_y,x\partial_x+\partial_y+1 \rangle\subset D_2$. This $D_2$-ideal has holonomic rank $2$ and the connection matrices for the choices of bases $\cO_1=\{1,\partial_x\}$ and $\cO_2=\{ 1,\partial_y\}$ are provided in \cite[Section~3.1]{ConnMatM2}. 
Via the border of $\cO_1$, $\partial \cO_1=\{\partial_x^2,\partial_x\partial_y,\partial_y\}$, we associate the $R_2$-ideal 
\begin{align*}
    J_{\cO_1} \,= \, \left\langle \partial_x^2+\frac{3x-y}{x(x-y)}\partial_x+\frac{1}{x(x-y)},\, \partial_x\partial_y+\frac{x+y}{y(y-x)}\partial_x+\frac{1}{y(y-x)},\, \partial_y+ \frac{x}{y}\partial_x + \frac{1}{y}\right\rangle .
\end{align*}  
Via the border of~$\cO_2$, $\partial \cO_2=\{\partial_y^2,\partial_x\partial_y,\partial_x\}$, we associate the $R_2$-ideal 
\begin{align*}%\label{ex:borderJ2}
    J_{\cO_2} \,= \, \left\langle \partial_y^2-\frac{3y-x}{y(y-x)}\partial_y+\frac{1}{y(y-x)},\, \partial_x\partial_y+\frac{x+y}{x(x-y)}\partial_y+\frac{1}{x(x-y)},\, \partial_x+ \frac{y}{x}\partial_y + \frac{1}{x}\right\rangle .
\end{align*}
The listed generators are border bases of the resulting $R_2$-ideals for $\cO_1$ and $\cO_2$, respectively. One can check that  $J_{\cO_1}$ and $J_{\cO_2}$ have the same Gröbner basis and hence indeed coincide. % as \mbox{$R_2$-ideals}.
\end{example}

We now comment on the construction of $J_\cO$ above.
\begin{remark}\label{rem:iterate_all_cOs}
    By \Cref{prop:groebnercorner}, an $R$-ideal $J$ admits an $\cO$-border basis for at least some order ideal~$\cO$. Thus, in the above construction, one can iterate over all possible order ideals $\cO$ of $m$ elements in $n$ variables---the counting of which is given by multidimensional partitions~\cite{bhatia1997asymptotic}---until one finds one order ideal for which the matrix $g$ has non-zero determinant. 
    Heuristically, we found that  order ideals $\cO$ containing derivatives of  the smallest possible orders are preferable for computation. I.e., they give rise to more easily invertible matrices $g$  and to generators of the ideal $J_{\cO}$ with simpler terms. 
\end{remark}
We will give more examples of this construction in Section \ref{sec:explappl}.

\subsection{Classification of certain \texorpdfstring{$D$}{D}-ideals}\label{sec:classFrobconst}
While the classification of general $D$-ideals of fixed holonomic rank is highly intricate, some special cases can be made explicit.
The Hilbert scheme of $m$ points can be interpreted to provide some parts of the moduli space of $D$-ideals of holonomic rank~$m$, namely the cases of PDEs with constant coefficients as well as that of Frobenius ideals. To be precise, our considerations apply to  ideals $R_nI$ in the rational Weyl algebra, as discussed already earlier. Since, however, they are the natural starting point, we here start from ideals $I\subset D_n$. 
\subsubsection{Constant coefficients}\label{sec:constcoeff}
Let $\mathfrak{I}\subset \CC[\partial_1,\ldots,\partial_n]\subset D_n$ encode a system of PDEs with constant coefficients, 
so that we are in the commutative setup of 
\Cref{sec:Hilb}. We tackle the  study of $D$-ideals $I=D_n\mathfrak{I}$ with $\mathfrak{I}\subset \CC[\partial]$ of a fixed holonomic rank $m$ using the Hilbert scheme of points~$\Hilb_n^m$. For this to work out, we need that the holonomic rank of $I$ and $\dim_{\CC} (\CC[\partial]/\mathfrak{I}) $ coincide.

\begin{lemma}\label{prop:rankconstant}
  Let $\mathfrak{I}\subset \CC[\partial]$ and denote $I=D_n\mathfrak{I}$. Then $\dim_{\CC}(\CC[\partial]/\mathfrak{I})=\rank(I)$.
\end{lemma}
\begin{proof}
We switch to the category of $\CC$-vector spaces and start from the short exact sequence of $\CC$-vector spaces $0\to \mathfrak{I} \to \CC[\partial]\to \CC[\partial]/\mathfrak{I} \to 0$. Tensorizing with $\CC(x)$ from the left yields the short exact sequence of $\CC(x)$-vector spaces $0\to R_nI \to R_n \to R_n/R_nI\to 0$. To see that, one exploits that $\CC(x)\otimes_{\CC} \mathfrak{I}=R_n\mathfrak{I}$ as $\CC(x)$-vector spaces, which becomes visible from normally ordered form of differential operators. In short, for constant coefficients, passing from the $\CC$-vector space underlying $\CC[\partial]/\mathfrak{I}$ to the $\CC(x)$-vector space underlying $R_n/R_nI$ is just a field extension by $\CC(x)$ and hence the  dimension of $R_n/R_nI$ over ${\CC(x)}$ equals the dimension of $\CC[\partial]/\mathfrak{I}$ over~$\CC$.
\end{proof}

In particular, any $\CC$-basis of $\CC[\partial]/\mathfrak{I}$ is a $\CC(x)$-basis of~$R_n/R_nI$. In these bases, the connection matrices have constant entries, and hence are pairwise commuting. We can therefore use the theory of Hilbert schemes of points for the classification of such ideals. 

\subsubsection{Frobenius ideals}\label{sec:Frob}
We now move on to Frobenius ideals. Denote by $\theta_i=x_i\partial_i$ the $i$-th Euler operator. A \mbox{$D_n$-ideal} $I$ is a {\em Frobenius ideal} if $I=D_n\mathfrak{I}$ for some ideal $\mathfrak{I}\subset \CC[\theta_1,\ldots,\theta_n].$ The polynomial ring $\CC[\theta_1,\ldots,\theta_n]$ is a commutative subring of the Weyl algebra, and we can again consider Hilbert schemes of points, with our variables now being the $\theta_i$'s.
The ideal $\mathfrak{I}$ is called {\em Artinian} if  $\CC[\theta]/\mathfrak{I}$ is finite-dimensional as a $\CC$-vector space.

\begin{proposition}[{\cite[Proposition 2.3.6]{SST00}}]
A Frobenius ideal $I=D_n\mathfrak{I}$ is holonomic if and only if the underlying $\CC[\theta]$-ideal $\mathfrak{I}$ is Artinian. In this case, $\rank(I)=\dim_{\CC}(\CC[\theta]/\mathfrak{I}).$ 
\end{proposition}
Any $\CC$-basis of $\CC[\theta]/\mathfrak{I}$ is a $\CC(x)$-basis of~$R_n/R_nI$. The basis elements are now polynomials in the $\theta$'s.
In any such basis, the connection matrices of $I=D_n\mathfrak{I}$ as in~\eqref{eq:connmatrices} are $A_i=\frac{1}{x_i}M_{\theta_i}^\top$, where $M_{\theta_i}$ is the matrix that describes the multiplication $\theta_i\colon \CC[\theta]/\mathfrak{I}\longrightarrow \CC[\theta]/\mathfrak{I}$. By construction, these matrices are pairwise commuting, since the $M_{\theta_i}$'s have entries in~$\CC$.

\begin{example}
    Let $P=\theta^2-\theta+1\in \CC[\theta]$. In the $\CC$-basis $(1,\theta)$ of $\CC[\theta]/\langle P \rangle$, the multiplication by $\theta$ is described by the matrix $M_{\theta}=\left(\begin{smallmatrix} 0 & -1 \\ 1 & 1\end{smallmatrix}\right)$. The $D_1$-ideal generated by $P$ has holonomic rank $2$. As a $\CC(x)$-basis of $R_1/R_1P$, we choose $(1,x\partial)$ so that for any $f\in \Sol(P)$, we have
\begin{align*}
    \begin{pmatrix}
        f \\ xf'
    \end{pmatrix}' \,=\, \begin{pmatrix}
        0 & \frac{1}{x}\\ -\frac{1}{x} & \frac{1}{x}
    \end{pmatrix}\cdot   \begin{pmatrix}
        f \\ xf'
    \end{pmatrix}
\end{align*}
by the Leibniz rule. Indeed, the connection matrix equals $\frac{1}{x}M_{\theta}^\top$.
\end{example}

\medskip
As explained above, in the case of Frobenius ideals~$I$, there always exists a $\mathbb{C}(x)$-basis of $R_n/R_nI$ for which the matrices $A_i$ are pairwise commuting, i.e., 
\begin{align*}
    \left[A_i,A_j\right] \,=\, 0 \quad \text{for all } \, i\neq j \, .
\end{align*}
For each tuple $i<j$, this yields $m^2$ quadratic equations in the entries of the connection matrices. %$m^{(i)}_{jk}$'s. 
Together with the choice of an order ideal, connection matrices of Frobenius ideals can hence be regarded as a subset of the variety of $n$-tuples of $m\times m$ matrices with entries in~$\CC$ that are pairwise commuting. The latter is a subvariety of the affine space~$\mathbb{M}_{m\times m}^n(\CC)$ of $n$-tuples of $m\times m$ matrices.
Over algebraically closed fields, such as the complex numbers or the field of Puiseux series in $\varepsilon$, 
\begin{align}\label{eq:puisex}
    \CC\{\!\{\varepsilon\}\!\} \,=\, \bigcup_{N>0}\CC(\!(\varepsilon^{1/N})\!) \, ,
\end{align}
%$\CC\{\! \{ \varepsilon \}\!\}$, % as in~\eqref{eq:puiseux}, 
varieties of commuting matrices are investigated in~\cite{VarCommMat} 
and they are inherently linked to Quot schemes via the ADHM construction~\cite{ADHM} named after Atiyah, Drinfel’d, Hitchin, and Manin, see also~\cite{QuotInt}. 

\section{Applications in physics}\label{sec:explappl}
In this section, we visit linear differential equations behind integrals in string theory, particle physics, and cosmology. We provide our Mathematica notebooks at \url{https://uva-hva.gitlab.host/universeplus/BorderBasesRationalWeylAlgebra}. In our computations, we use the {\tt HolonomicFunctions} package~\cite{HolFun}. 
%via the GitLab of \href{https://positive-geometry.com/}{UNIVERSE+}  
Because Feynman integrals $\mathcal{I}(x)$ are holonomic functions in the kinematic variables~$x_i$, they satisfy a holonomic system of PDEs, and hence in particular ordinary linear differential equations of finite order $r_i\in \mathbb{N}_{>0}$ in each of the $\partial_i$'s, see~\cite[Proposition 2.10]{GLS21}. We denote the corresponding differential operator of order $r_i$ by $P_i\in \CC[x_1,\ldots,x_n]\langle \partial_i\rangle$. Whenever the corresponding $r_i$ is minimal, the operator is called a {\em Picard--Fuchs operator of order $r_i$}~\cite{PicardFuchsEquations}, and one can read off this operator from the Pfaffian system associated to a Feynman integral {\cite[Section~7.1.3]{WeinzierlFeynmanIntegralsBook}}. Note that the order of the Picard--Fuchs operator is bounded above by the number of master integrals.

\subsection{A string integral}
We revisit the string integral from \Cref{ex:stringy} in two variables. We first perform a change from $F=(F_1,F_2,F_3)^\top$ to the vector $\widetilde{F}=(F_1,\partial_1\bullet F_1,\partial_2\bullet F_1)^\top$ of %$F_1$ and 
first-order partial derivatives of~$F_1$ so that the associated order ideal is going to be $\mathcal{O}_1\coloneqq \{ 1,\partial_{1},\partial_2\}$. From the first rows of the connection matrices $A_1,A_2$ as in~\eqref{eq:A3A4stringy}, one reads that $\widetilde{F}=g_1\cdot F$~for
\begin{align}
    g_1 \,=\, \begin{pmatrix}
        1 & 0 & 0\\
        \frac{s_{12}+s_{23}}{x_1} & - \frac{s_{12}}{x_1} & 0\\
        \frac{s_{24}}{x_2} & \frac{s_{12}}{x_2} & - \frac{s_{12}}{x_2}
    \end{pmatrix} \, \in \,  \mathbb{M}_{3\times 3}\left(\CC(s_{12},s_{23},s_{24})(x_1,x_2)\right) .
\end{align}
Since $s_{12}\neq 0$ in $\CC(s_{12},s_{23},s_{24})$, the determinant of $g_1$ is non-zero. In applications, when plugging in fixed values for $s$, one would hence need to choose non-zero~$s_{12}$. Hence, $\widetilde{F}$ is a basis and we can use $g_1$ to carry out a gauge transformation of the connection matrices.  We denote the resulting matrices by $\widetilde{A}_1$ and $\widetilde{A}_2$. To read a border basis for the order ideal $\mathcal{O}_1=\{ 1,\partial_{1},\partial_{2}\}$, we need to look at the border elements $\partial_1^2, \partial_1\partial_2,\partial_2^2$.
From the second row of $\widetilde{A}_1$ and the second and third row of~$\widetilde{A}_2$, we read the three differential operators 
\begin{align}\begin{split}\label{eq:P1P2P3stringy}
    P_1^{(1)} & \,=\, \partial_1^2 \,+\, \frac{s_{23}(s_{12} + s_{23} + s_{24} + s_{25})}{x_1(x_1-1)} \,+\, \frac{s_{23} x_2(x_2-1)}{x_1(x_1-1) (x_1 - x_2)} \partial_2 \\ & \quad  - \frac{1}{x_1( x_1-1) (x_1 - 
   x_2)} \big((-1 + s_{12} + 2 s_{23} + s_{24} + s_{25}) x_1^2 + (-1 + s_{12} + s_{23}) x_2 \\ & \quad - 
  x_1 (-1 + s_{12} + s_{23} + s_{24} - x_2 + s_{12} x_2 + 2 s_{23} x_2 + s_{25} x_2)\big) \partial_1  , \\ 
  P_2^{(1)} & \,=\, \partial_1\partial_2 \,+\, \frac{s_{24}}{x_1 - x_2}\partial_1 \, -\, \frac{s_{23}}{x_1 - x_2}\partial_2 , \\ 
  P_3^{(1)} &\,=\, \partial_2^2 \,+\, \frac{s_{24} (s_{12} + s_{23} + s_{24} + s_{25})}{( x_2-1) x_2}  \,-\, \frac{
 s_{24} (x_1-1) x_1}{x_2(x_2-1)(x_1 - x_2)} \partial_1 \\ &  \quad - 
\left(s_{24} \left( \frac{1}{x_2-1} + \frac{1}{x_2}\right) + \frac{s_{25}}{x_2-1} - \frac{1}{x_2} + \frac{s_{12}}{x_2} - 
    \frac{s_{23}}{x_1 - x_2} \right)\partial_2.
\end{split}\end{align}
The operators $P_1^{(1)},P_2^{(1)},P_3^{(1)}$ are an $\mathcal{O}_1$-border basis of the $R_2$-ideal $J_{\cO_1}$ generated by them, where $R_2(s)=R_2\otimes \CC(s)$ is the rational Weyl algebra in the variables~$x_1,x_2$ with parameters~$s$. We approved computationally that, for generic $s_{ij}$'s, $P_1^{(1)},P_2^{(1)},P_3^{(1)}$ are a Gröbner basis w.r.t.\ both the degree revlex and degree lex built on $\partial_1\succ \partial_2$, in coherence with \Cref{prop:groebnercorner} and using that the first border coincides with the corners in this case. 

\medskip

In the flavor of Picard--Fuchs operators in the direction of~$x_1$, we could also have chosen the order ideal $\mathcal{O}_2=\{ 1,\partial_1,\partial_1^2\}$ instead. The respective gauge matrix then is 
\begin{comment}
{\tiny
\begin{align*}
    %g_2 \,=\, 
    \begin{pmatrix} 
    1 & 0& 0\\
    \frac{s_{12}+s_{23}}{x_1} & -\frac{s_{12}}{x_1} & 0 \\
   \frac{{s_{12}}^2+{s_{12}} \left(2 {s_{23}}+\frac{{s_{24}}
   x_1}{x_1-x_2}+\frac{{s_{25}} x_1}{x_1-1}-1\right)+({s_{23}}-1)
   {s_{23}}}{x_1^2} & \frac{{s_{12}} \left(\frac{-x_1
   ({s_{12}}+{s_{23}}+{s_{25}})+{s_{12}}+{s_{23}}+x_1-1}{x_1-1}+\frac{x_1
   ({s_{23}}+{s_{24}})}{x_2-x_1}\right)}{x_1^2} & \frac{{s_{12}} {s_{23}}
   (x_2-1)}{(x_1-1) x_1 (x_1-x_2)}
 \end{pmatrix}.
\end{align*}
}
\end{comment}
\begin{align*}
    g_2 \,=\, 
    \begin{pmatrix} 
1 & \frac{s_{12}+s_{23}}{x_1} &  \frac{{s_{12}}^2+{s_{12}} \left(2 {s_{23}}+\frac{{s_{24}}
   x_1}{x_1-x_2}+\frac{{s_{25}} x_1}{x_1-1}-1\right)+({s_{23}}-1)
   {s_{23}}}{x_1^2} \\
0 & -\frac{s_{12}}{x_1} & \frac{{s_{12}} \left(\frac{-x_1
   ({s_{12}}+{s_{23}}+{s_{25}})+{s_{12}}+{s_{23}}+x_1-1}{x_1-1}+\frac{x_1
   ({s_{23}}+{s_{24}})}{x_2-x_1}\right)}{x_1^2} \\
0 & 0 & \frac{{s_{12}} {s_{23}}
   (x_2-1)}{(x_1-1) x_1 (x_1-x_2)}
 \end{pmatrix}^\top .
\end{align*}

One now has $\partial \cO_2=\{ \partial_2,\partial_1\partial_2,\partial_1^2\partial_2,\partial_1^3\}$ and the four resulting differential operators in the $\cO_2$-border basis are
\begingroup
\allowdisplaybreaks
\begin{align*}
    P^{(2)}_1 &\,=\, \partial_2 +\frac{(s_{12}+s_{23}+s_{24}+s_{25})(x_1-x_2)}{x_2(x_2-1)}+\frac{x_1(x_1-1)(x_1-x_2)}{s_{23} x_2(x_2-1)}\partial_1^2  \\
    &\, \phantom{=}+ \frac{1}{{s_{23}} {x_2}({x_2}-1) } 
    \bigg[ {x_2}-{x_1}^2 ({s_{12}}+2 {s_{23}}+{s_{24}}+{s_{25}}-1)  \\
    &\, \phantom{=+\bigg[}{+x_1} ({x_2} ({s_{12}}+2 \
{s_{23}}+{s_{25}}-1)+{s_{12}}+{s_{23}}+{s_{24}}-1)-{x_2} \
({s_{12}}+{s_{23}}) \bigg] \partial_1, \\
    P^{(2)}_2 &\,=\, \partial_1\partial_2 +\frac{{s_{23}} ({s_{12}}+{s_{23}}+{s_{24}}+{s_{25}})}{({x_2}-1) {x_2}}+\frac{{x_1}({x_1}-1) }{{x_2}({x_2}-1) } \partial_1^2  \\
    &\,\phantom{=} -\frac{{s_{12}} ({x_1}-1)+{s_{23}} (2 {x_1}-1)+{s_{24}} \
{x_1}+{s_{24}} {x_2}-{s_{24}}+{s_{25}} {x_1}-{x_1}+1}{({x_2}-1) {x_2}} \partial_1 ,\\
    P^{(2)}_3 &\,=\, \partial_1^2\partial_2 +\frac{({s_{23}}-1) {s_{23}} \
({s_{12}}+{s_{23}}+{s_{24}}+{s_{25}})}{({x_2}-1) {x_2} ({x_1}-{x_2})}  \\
&\,\phantom{=} -\frac{({s_{23}}-1) ({s_{12}} ({x_1}-1)+{s_{23}} (2 \
{x_1}-1)+{s_{24}} {x_1}+{s_{24}} {x_2}-{s_{24}}+{s_{25}} \
{x_1}-{x_1}+1)}{({x_2}-1) {x_2} ({x_1}-{x_2})} \partial_1\\
&\, \phantom{=}+\frac{({s_{23}}-1) {x_1}^2-({s_{23}}-1) {x_1}+{s_{24}} ({x_2}-1) {x_2}}{({x_2}-1) {x_2} ({x_1}-{x_2})} \partial_1^2\
, \\
P^{(2)}_4 &\,=\, \partial_1^3 -\frac{({s_{23}}-1) {s_{23}} \
({s_{12}}+{s_{23}}+{s_{24}}+{s_{25}})}{({x_1}-1) {x_1} ({x_1}-{x_2})} \\
&\,\phantom{=}+\frac{s_{23}-1}{x_1(x_1-1)(x_1-x_2)}\bigg[2 {s_{24}} {s_{12}} (2 {x_1}-{x_2}-1)+2 {s_{24}} {x_1}-{s_{24}}+2 {s_{25}} \
{x_1}-2 {x_1}+1 \\
&\, \phantom{=+\bigg[}{+s_{23}} (3 {x_1}-{x_2}-1)-{s_{25}} {x_2}+{x_2} \bigg] \partial_1  \\
&\,\phantom{=}+\bigg[ \frac{{s_{23}}+{s_{24}}-1}{{x_2}-{x_1}} +\frac{-{x_1} ({s_{12}}+2 \
{s_{23}}+{s_{25}}-3)+{s_{12}}+{s_{23}}-2}{({x_1}-1) {x_1}} \bigg] \partial_1^2 .
\end{align*}
\endgroup
Here, $P^{(2)}_4$ is a Picard--Fuchs operator of order 3 in the $x_1$-direction.
We computed in Mathematica that, for generic  $s_{ij}$'s, a Gröbner basis of the $R_2(s)$-ideal generated by the ``corners'' $P^{(2)}_1$ and $P^{(2)}_4$ with respect to the degree revlex order built on $\partial_1 \succ \partial_2$ is given by the operators $P_1^{(1)},P_2^{(1)},P_3^{(1)}$ in~\eqref{eq:P1P2P3stringy}. Moreover, $\langle P^{(2)}_1,P^{(2)}_2,P^{(2)}_3,P^{(2)}_4\rangle = \langle P^{(2)}_1,P^{(2)}_4\rangle$. Taking a Gröbner basis with respect to lex for $\partial_1 \succ \partial_2$ spits out two operators (with standard monomials $\{1,\partial_2,\partial_2^2\}$), one of them being a Picard--Fuchs operator of order $3$ in the $x_2$-direction.  We also note that $P^{(2)}_1,P^{(2)}_4$ are a reduced Gröbner basis for the lexicographic order based on $\partial_2\succ \partial_1$ and that $P_1^{(1)},P_2^{(1)},P_3^{(1)}$ are also a Gröbner basis for degree revlex for~$\partial_2 \succ \partial_1$.

\subsection{Sunrise Feynman integral}\label{sec:sunrise}
We here consider the unequal mass sunrise Feynman integral. %This integral 
It belongs to the family of so-called banana integrals, and corresponds to a graph with $\ell=2$ independent cycles and $\ell+2=4$ variables, $x_i=m_i^2$, for $i=1,2,\ldots,\ell+1$, and $x_0=s$. In the physics context, $s$ measures the rest mass of the external particle, and $m_i$ is the mass of the $i$-th massive internal leg, see e.g.\ the Introduction of \cite{Elvang_Huang_2015}. We are moreover setting the reference mass scale to 1, i.e., $\mu^2=1$. We start from the system of PDEs in matrix form as in \cite[Eq.~(4.10)]{Maggio:2025jel}, but keep the variable~$x_0$:
\begin{align}
\label{eq:SunriseRModule}
\partial_i \bullet \mathcal{I} \,=\, A_i(d;x_0,x_1,x_2,x_3) \cdot \mathcal{I}\, , \quad i = 0,1,2,3,
\end{align}
where $d$ is the dimension of Minkowski spacetime and $\mathcal{I}$ is a vector of $7$ master integrals, which we denote as $\mathcal{I} =
\left(\mathcal{I}_1 ,
\mathcal{I}_2,
\ldots,
\mathcal{I}_7 \right)^\top$. 
The connection matrices $A_i(d;x_0,x_1,x_2,x_3) \in \mathbb{M}_{7\times 7}(\mathbb{C}(d)(x_0,x_1,x_2,x_3))$ can be found in the ancillary file {\tt sunrise.nb} of \cite{Maggio:2025jel} for $i=1,2,3$, and were kindly provided to us by Yoann Sohnle including that for the variable~$x_0$, which he computed using the Mathematica package {\tt LiteRed}~\cite{LiteRed}.  \Cref{eq:SunriseRModule} describes the structure of an $R_4(d)$-module.
We remark that the first three integrals, $\cI_1,\cI_2,\cI_3$, are called \textit{tadpole} integrals in the physics literature, and cannot correspond to the cyclic generators of this $R_4(d)$-module. The master integral $\cI_4$ corresponds to the sunrise Feynman integral, and will turn out to give rise to a cyclic generator of this $R_4(d)$-module.
  
{Having tried several candidates for possible order ideals, we found $\cO=\{ 1,\partial_1,\partial_2,\partial_3,\partial_1\partial_3,\partial_2\partial_3,\partial_3^2\}$ to work nicely. One can check that
\begin{align}\label{ex:gaugesunrise}
\begin{pmatrix}
    \cI_4 \\
    \partial_1\bullet \cI_4\\
    \partial_2 \bullet \cI_4\\
    \partial_3 \bullet \cI_4\\
    \partial_1\partial_3 \bullet \cI_4\\
    \partial_2\partial_3 \bullet \cI_4\\
    \partial_3^2 \bullet \cI_4
\end{pmatrix}
\,=\,
g_{\text{sunrise}}\cdot \begin{pmatrix}
    \cI_1 \\
    \cI_2\\
    \cI_3\\
    \cI_4\\
    \cI_5\\
    \cI_6\\
    \cI_7
\end{pmatrix} ,
\end{align}
where $g_{\textrm{sunrise}}\in \mathbb{M}_{7\times 7}(\mathbb{C}(d)(x_0,x_1,x_2,x_3))$ is an invertible matrix. Therefore, this order ideal gives rise to a basis of the corresponding quotient module. 
We now use $g_{\textrm{sunrise}}$ to perform a gauge transformation on the connection matrices $A_0,\ldots,A_3$ to the basis $\cO \bullet \cI_4$.
This results in connection matrices from which we can directly read off a border basis of an annihilating $R_4(d)$-ideal of~$\cI_4$.}
The border of $\cO$ consists of $16$ elements, namely
$ \partial \cO = \{ \partial_0,\partial_0\partial_1,\partial_0\partial_2,\partial_0\partial_3,\partial_0\partial_1\partial_3,\partial_0\partial_2\partial_3,\partial_0\partial_3^2,\partial_1^2,\partial_1\partial_2,\partial_1\partial_2\partial_3,\partial_1^2\partial_3,\partial_1\partial_3^2,\partial_2^2,\partial_2^2\partial_3,\partial_2\partial_3^2,\partial_3^3\},$  
to which we associate $16$ differential operators. We denote the resulting $R_4(d)$-ideal by~$J_{\cO}$ and compare it to the annihilating $D$-ideal of the associated Feynman integrals of generic mass banana integrals in dimensional regularization as provided in the recent article~\cite{Didealbanana}. 
The respective Weyl algebra is $D=D_4(d)=\CC(d)[x_0,x_1,x_2,x_3]\langle \partial_0,\partial_1,\partial_2,\partial_3\rangle$, and the operators from \cite[Theorem 1]{Didealbanana} read~as
\begingroup
\allowdisplaybreaks
\begin{align}\begin{split}\label{eq:Didealbanana}
P_E &\,=\, x_0\partial_0+x_1\partial_1+x_2\partial_2+x_3\partial_3+3-d ,\\ 
P_1 &\,=\,  -x_0\partial_0^2 + x_1\partial_1^2 - \frac{d}{2}\partial_0 + \frac{4-d}{2}\partial_1,\\
P_2 &\,=\, -x_0\partial_0^2 + x_2\partial_2^2 - \frac{d}{2}\partial_0 + \frac{4-d}{2}\partial_2, \\
P_3 &\,=\, -x_0\partial_0^2 + x_3\partial_3^2 - \frac{d}{2}\partial_0 +\frac{4-d}{2}\partial_3, \\
P_s &\,=\, \partial_1\partial_2\partial_3 + \partial_0\partial_2\partial_3 + \partial_0\partial_1\partial_3 + \partial_0\partial_1\partial_2.
\end{split}\end{align}
\endgroup
The $D$-ideal $I=\langle P_E,P_1,P_2,P_3,P_s\rangle$ generated by them has holonomic rank $7$ and its singular locus is cut out by a single polynomial. 
For the degree reverse lexicographic order built on $\partial_0\succ \partial_1\succ \partial_2 \succ \partial_3$, the standard monomials are $\{ 1,\partial_1,\partial_2,\partial_3,\partial_1\partial_3,\partial_2\partial_3,\partial_3^2\}$, coinciding with the order ideal $\cO$ considered above. 

By comparing the Gröbner bases of the $R_4(d)$-ideals $J_{\cO}$ and $R_4(d)I$, we obtain that \mbox{$J_{\cO}=R_4(d)I$.} The Gröbner basis of $R_4(d)I$ with respect to degree reverse lex based on $\partial_0 \succ \partial_1 \succ \partial_2 \succ \partial_3$ consists of seven elements, namely exactly the border basis elements marked by the corners of $\mathcal{O}$, which are $\partial_0,\partial_1^2,\partial_1\partial_2,\partial_1\partial_3^2,\partial_2^2,\partial_2\partial_3^2,$ and $\partial_3^3$.

\subsection{Cosmological \texorpdfstring{$2$}{2}-site graph}
In cosmology, correlation functions are used to understand the distribution of matter and energy 
throughout the universe.
For the correlator of the $2$-site graph as a cosmological toy model, a matrix PDE in $\varepsilon$-factorized form was provided in~\cite[(1.9)]{DEcosmological}, see \Cref{sec:eps} for a definition. Therein, a basis of master integrals is constructed from the canonical forms of the underlying hyperplane arrangement, in the sense of positive geometry~\cite{PosGeom}.

For our undertaking to derive a border basis of an underlying ideal in the rational Weyl algebra $R_3(\varepsilon)=\CC(\varepsilon)(X_1,X_2,Y)\langle \partial_{X_1},\partial_{X_2},\partial_Y\rangle$, a basis consisting of monomials in the $\partial$'s is required. We change basis to the order ideal $\mathcal{O}_1=\{1,\partial_{Y},\partial_{Y}^2,\partial_{Y}^3\}$. The resulting gauge transformed connection matrices are not in \mbox{$\varepsilon$-factorized} form anymore, but they have several advantages: we can directly read an annihilating $D$-ideal of the correlator of the expected holonomic rank as well as a Picard--Fuchs operator.
The border of $\cO_1$ is $\partial \cO_1=\{ \partial_{X_1},\partial_{X_1}\partial_{Y},\partial_{X_1}\partial_{Y}^2,\partial_{X_1}\partial_{Y}^3,\partial_{X_2},\partial_{X_2}\partial_{Y},\partial_{X_2}\partial_{Y}^2,\partial_{X_2}\partial_{Y}^3,\partial_{Y}^4\}$ so that one can read an $\cO_1$-border basis $\{P_i^{(1)}|i=1,\ldots,9\}$  of an $R_3(\varepsilon)$-ideal $J_{\cO_1}$ of holonomic rank~$4$.  One of the $P_i^{(1)}$'s, namely the one marked by $\partial_{Y}^4$, $P_9^{(1)}\in \CC(\varepsilon)(X_1,X_2,Y)\langle \partial_Y\rangle$, is a Picard--Fuchs operator of the correlation function of order 4 in the $Y$-direction. The ideal is generated by the operators $P_1^{(1)},P_5^{(1)},P_9^{(1)}$ arising from the corners of~$\cO_1$, $\{\partial_{X_1},\partial_{X_2},\partial_Y^4\}$. These three operators are a Gröbner basis for the lexicographic order based on $\partial_{X_1} \succ \partial_{X_2} \succ \partial_{Y}$. 

Moreover, we checked computationally that $J_{\cO_1}$ coincides with the $R_3(\varepsilon)$-ideal generated by the generators of the $D$-ideal $I$ given in~\cite[(11)]{FPSW25}, which annihilates the cosmological correlator. In \cite{FPSW25}, a different $\CC(\varepsilon)(X_1,X_2,Y)$-basis of $R_3(\varepsilon)/R_3(\varepsilon)I$ was used for writing the connection matrices, namely the order ideal $\mathcal{O}_2=\{ 1,\partial_{X_1},\partial_{X_2},\partial_{X_1}\partial_{X_2} \} $. We point out that none of $\{1,\partial_{X_1},\partial_{X_1}^2,\partial_{X_1}^3\}$, $\{1,\partial_{X_2},\partial_{X_2}^2,\partial_{X_2}^3\}$, and $\{1,\partial_{X_1},\partial_{X_2},\partial_{Y}\}$ is a basis of $R_3(\varepsilon)/R_3(\varepsilon)I$. Hence, there are no Picard--Fuchs operators of order 4 in the directions of $X_1$ or~$X_2$.

The border of $\cO_2$ is
\begin{align}\label{ex:bordertwosite}
    \partial \cO_2 \,=\, \left\{ \partial_{X_1}^2,\partial_{X_1}^2\partial_{X_2},\partial_{X_2}^2,\partial_{X_1}\partial_{X_2}^2,\partial_{Y},\partial_{X_1}\partial_{Y},\partial_{X_2}\partial_{Y},\partial_{X_1}\partial_{X_2}\partial_{Y} \right\}
\end{align}
and its corners are $\{\partial_{Y},\partial_{X_1}^2,\partial_{X_2}^2\}$.
For each border element, from the resulting connection matrices, we read one border basis element $P_i^{(2)}$, $i=1,\ldots,8$, that is  marked by the respective border element $b_i$ in~\eqref{ex:bordertwosite}:
\begingroup
\allowdisplaybreaks
\begin{align*}
P_1^{(2)} \,&=\,\partial_{X_1}^2-\frac{\varepsilon(1 -2 \varepsilon)}{X_1^2-Y^2}
-\frac{(X_1 (3 \varepsilon 
-2)+Y \varepsilon )}{(X_1-Y) (X_1+Y)}\partial_{X_1}
-\frac{ \varepsilon  
(X_2+Y)}{(X_1-Y) (X_1+Y)}\partial_{X_2}\\
\,&\phantom{=\,}+\frac{ (X_2+Y)}{X_1-Y}\partial_{X_1}\partial_{X_2}
,\\
P_2^{(2)} \,&=\,\partial_{X_1}^2\partial_{X_2}-\frac{\varepsilon ^2(1-2 \varepsilon)}{(X_1+X_2) (X_1-Y) (X_1+Y)}-\frac{ \varepsilon  (2 \varepsilon -1)}{(X_1+X_2) (X_1-Y)}\partial_{X_1} \\
\,&\phantom{=\,}+\frac{\varepsilon  (X_1 (\varepsilon -1)-X_2 \varepsilon -2 Y \
\varepsilon +Y)}{(X_1+X_2) (X_1-Y) (X_1+Y)}\partial_{X_2}\\
\,&\phantom{=\,}+ \frac{-2 X_1^2 (\varepsilon -1)+X_1 (X_2+Y (2 \varepsilon -1))+Y (2 \
\varepsilon -1) (X_2+Y)}{(X_1+X_2) (X_1-Y) (X_1+Y)}\partial_{X_1}\partial_{X_2}, \\
P_3^{(2)} \,&=\,\partial_{X_2}^2 -\frac{\varepsilon  (X_1+Y)}{(X_2-Y) (X_2+Y)}\partial_{X_1}
-\frac{X_2 (3 \varepsilon -2)+Y \varepsilon}{(X_2-Y) \
(X_2+Y)}\partial_{X_2} 
+\frac{ (X_1+Y)}{X_2-Y}\partial_{X_1}\partial_{X_2},\\
P_4^{(2)} \,&=\, \partial_{X_1}\partial_{X_2}^2 
-\frac{\varepsilon ^2(1-2 \varepsilon)}{(X_1+X_2) (X_2-Y) (X_2+Y)}
-\frac{ \varepsilon  (2 \varepsilon -1)}{(X_1+X_2) (X_2-Y)}\partial_{X_2}
\\
\,&\phantom{=\,}-\frac{\varepsilon  (X_1 \varepsilon +X_2 (-\varepsilon )+X_2+Y \
(2 \varepsilon -1))}{(X_1+X_2) (X_2-Y) (X_2+Y)}\partial_{X_1}
\\
\,&\phantom{=\,}+\frac{X_1 (X_2+Y (2 \varepsilon -1))-2 X_2^2 (\varepsilon -1)+X_2 Y \
(2 \varepsilon -1)+Y^2 (2 \varepsilon -1)}{(X_1+X_2) (X_2-Y) (X_2+Y)}
\partial_{X_1}\partial_{X_2},
\\
P_5^{(2)} \,&=\, \partial_{Y}
-\frac{2 \varepsilon }{Y}
+\frac{ X_1}{Y}\partial_{X_1} 
+\frac{ X_2}{Y}\partial_{X_2}   ,\\
P_6^{(2)} \,&=\, \partial_{X_1}\partial_Y 
-\frac{ \varepsilon  (2 \varepsilon -1)X_1}{Y (X_1-Y) (X_1+Y)}
+\frac{X_1^2 (\varepsilon -1)+X_1 Y \varepsilon +Y^2 (2 \varepsilon \
-1)}{Y \left(X_1^2-Y^2\right)}
\partial_{X_1} \\
\,&\phantom{=\,}
-\frac{  \varepsilon  X_1(X_2+Y)}{Y (Y-X_1) (X_1+Y)}\partial_{X_2}
-\frac{ (X_1+X_2)}{X_1-Y} \partial_{X_1}\partial_{X_2}
,\\
P_7^{(2)} \,&=\,  \partial_{X_2}\partial_{Y} 
-\frac{ \varepsilon  X_2(2 \varepsilon -1)}{Y (X_2-Y) (X_2+Y)}
-\frac{X_2 \varepsilon  (X_1+Y)}{Y (Y-X_2) (X_2+Y)}\partial_{X_1}\\
\,&\phantom{=\,}
+\frac{X_2^2 (\varepsilon -1)+X_2 Y \varepsilon +Y^2 (2 
\varepsilon -1)}{Y \left(X_2^2-Y^2\right)}
\partial_{X_2}
-\frac{(X_1+X_2)}{X_2-Y} \partial_{X_1}\partial_{X_2}
,\\
P_8^{(2)} \,&=\, \partial_{X_1}\partial_{X_2}\partial_{Y}+\frac{\varepsilon ^2 (2 \varepsilon -1) \left(X_1 X_2-Y^2\right)}{Y \
(X_1-Y) (X_1+Y) (Y-X_2) (X_2+Y)}\\
\,&\phantom{=\,}+\frac{\varepsilon  \left(X_1 X_2 \varepsilon +X_2 Y (\varepsilon \
-1)+Y^2 (1-2 \varepsilon )\right)}{Y (Y-X_1) (Y-X_2) (X_2+Y)} \partial_{X_1}\\
\,&\phantom{=\,}+\frac{\varepsilon  \left(X_1 X_2 \varepsilon +X_1 Y (\varepsilon \
-1)+Y^2 (1-2 \varepsilon )\right)}{Y (Y-X_1) (X_1+Y) (Y-X_2)}\partial_{X_2} \\
\,&\phantom{=\,}+\frac{\left(Y^2-X_1 X_2\right) (X_1 (X_2+Y (2 \varepsilon -1))+Y (X_2 \
(2 \varepsilon -1)+Y))}{Y (Y-X_1) (X_1+Y) (Y-X_2) (X_2+Y)} \partial_{X_1}\partial_{X_2}. 
\end{align*}
\endgroup
These eight operators generate an $R_3(\varepsilon)$-ideal $J_{\cO_2}$ of holonomic rank~$4$. By \Cref{prop:invgauge}, we get that $R_3(\varepsilon)J_{\cO_1}=R_3(\varepsilon)J_{\cO_2}$. 
\begin{remark}
The $P_i^{(2)}$'s are an $\cO_2$-border basis of~$J_{\cO_2}$, but for orders such as the (reverse) lexicographic or degree (reverse) lexicographic order, the operators $P_1^{(2)}$~and~$P_3^{(2)}$ can not be part of a Gröbner basis such that the terms written first are the leading terms: since in both operators, the term $\partial_{X_1}\partial_{X_2}$ occurs with a non-zero coefficient, 
%and, w.l.o.g., $\partial_{X_1}\succ \partial_{X_2}$, %or $\partial_{X_2}\succ \partial_{X_1}$, 
the two terms $\partial_{X_1}^2$ and $\partial_{X_2}^2$
can not both be larger than $\partial_{X_1}\partial_{X_2}$.
\end{remark}

\subsection{\texorpdfstring{$\varepsilon$}{eps.}-factorized form}\label{sec:eps}
The considered $R_n$-ideals in this section depend on an additional parameter~$\varepsilon$, such as a small parameter in the setup of dimensional regularization of Feynman integrals. The connection matrices $A_i \in \mathbb{M}_{m\times m}(\mathbb{C}(\varepsilon)(x))$ are said to be in {\em $\varepsilon$-factorized} form if they can be written as
\begin{align}
A_i \,=\, \varepsilon \cdot B_i \, , \quad i=1,\ldots,n,
\end{align}
with $B_i \in \mathbb{M}_{m\times m}(\mathbb{C}(x))$ independent of $\varepsilon$. 
The $\varepsilon$-factorized form of the connection matrices is particularly popular among Feynman integral practitioners for several reasons. Its main benefit is that it allows for the algorithmic computation of the (coefficients of) Feynman integrals in dimensional regularization in terms of iterated integrals, see~\cite[Section~6.3.3]{WeinzierlFeynmanIntegralsBook}. If such an $\varepsilon$-factorized form of a connection matrix exists\footnote{There are Feynman integrals for which one needs to go beyond the field of rational functions $\mathbb{C}(x)$ for the $\varepsilon$ to factor out, using algebraic or even transcendental functions of~$x$, see for instance \cite{Adams:2018yfj}. Geometrically, using a $\sqrt{x}$ would for instance require to pull back the considered $D$-module via a covering map that is ramified of degree~$2$. This goes beyond the scope of this work.}, software such as CANONICA~\cite{CanonicaPractice} implement heuristics to~find~it. 

If a connection matrix of an $R_n(\varepsilon)$-ideal $J$ is in $\varepsilon$-factorized form, this has further impacts on the their shape. In fact, one can show the following, see~\cite[Equation (5)]{henn2013multiloop}.

\begin{proposition}\label{prop:epsclosed}
Let $A$ be a connection matrix that is $\varepsilon$-factorized. Then  $\d A =0$.
\end{proposition}
Phrased in terms of the matrices $A_i$ in the Pfaffian system of~$J$, the condition $\d A=0$ translates as $[A_i,A_j]=0$ for any $i,j=1,\ldots,n$. 
\begin{proof}
Consider the integrability conditions of~\Cref{eq:integrabilityConditions}. If all the connection matrices are $\varepsilon$-factorized, then the left-hand side and right-hand side of the integrability conditions have different powers of $\varepsilon$, respectively. Thus, by equating the coefficients, one sees that all of them have to vanish.
\end{proof}
 
Being $\varepsilon$-factorized strongly depends on the chosen basis of~$R_n(\varepsilon)/J$. A gauge transformation w.r.t.\ $g\in \GL_m(\CC(\varepsilon ))$
%$g\in \GL_m(\CC\{ \! \{ \varepsilon \} \! \}$ 
keeps the $\varepsilon$-factorized form, while transformations w.r.t.\ $g\in \GL_n(\CC(\varepsilon)(x_1,\ldots,x_n))$ can easily break~it.
Since by \Cref{prop:epsclosed} being $\varepsilon$-factorized implies that 
the $A_i,A_j$'s are pairwise commuting, the set of $\varepsilon$-factorized connection matrices is a subset of the variety of commuting matrices. 

The variety $\mathbb{M}_{m\times m}^n(\CC)\cong \mathbb{A}_{\CC}^{m^2\cdot n}$ consists of $n$-tuples $(M_1,\ldots,M_n)$ of $m\times m$ matrices with entries in $\CC$, and we denote their entries by 
\begin{align*}
M_i \,=\, \left(m^{(i)}_{jk}\right)_{j,k=1,\ldots,m} .
\end{align*} 
The $m_{jk}^{(i)}$'s serve as the coordinates for~$\mathbb{M}_{m\times m}^n(\CC)$. The set of $n$-tuples of pairwise commuting matrices is a subvariety of $\mathbb{M}_{m\times m}^n$, cut out by quadratic equations in the entries of the matrices encoded by imposing that the matrices commute.

\begin{example}[$n=2,m=3$]\label{ex:comm23}
The following code in {\em Macaulay2} encodes the subvariety, {\tt C}, of $\mathbb{M}^2_{3\times 3}(\CC)$ consisting of commuting pairs of $3\times 3$ matrices $M=(m_{ij})_{i,j}$ and $N=(n_{ij})_{i,j}$.

\smallskip
\begin{small}
\begin{verbatim}
i1 : R = QQ[m_11..m_13,m_21..m_23,m_31..m_33,n_11..n_13,n_21..n_23,n_31..n_33];
i2 : M = matrix {{m_11..m_13},{m_21..m_23},{m_31..m_33}};
i3 : N = matrix {{n_11..n_13},{n_21..n_23},{n_31..n_33}};
i4 : commMN = M*N - N*M; L = flatten entries commMN; C = ideal L;
i5 : dim C, degree C
o5 = (13, 31)
\end{verbatim}
\end{small}
One obtains that the dimension of {\tt C} is $13$, its degree is $31$. 
\end{example}

\section{Conclusion}
Border bases generalize Gröbner bases for zero-dimensional ideals in polynomial rings. In this article, we transferred the notion of border bases to ideals of finite holonomic rank in the non-commutative rational Weyl algebra. We presented a characterization of border bases in terms of integrability conditions for the connection matrices and adapted the algorithms, such as the border division algorithm, to this non-commutative setup.

As an application of border bases, we presented how to systematically represent the \mbox{$R_n$-ideal} which underlies a system of linear PDEs that is given in matrix form. This is also of particular interest in current undertakings in the study of scattering amplitudes.

We also shed first light on the explicit description of parts of the moduli space $D$-ideals of fixed holonomic rank. More specifically, we addressed the classification of Frobenius ideals as well as $D$-ideals arising from PDEs with constant coefficients, building on the combinatorial description of the Hilbert scheme of points in affine space from the commutative setup. It will be worthwhile to investigate if this combinatorial approach can help to understand the geometry of the moduli space of integrable connections, which is highly intricate already in the one-dimensional case, see e.g.~\cite{Boalch2012,vdPS03}.

In future work, we plan to investigate the set of $n$-tuples of matrices 
\begin{align*}%\label{eq:An}
 %\cA_n  & \, \coloneqq \,  
 \left\{ (A_1,\ldots,A_n) \, | \, [A_i,A_j] \,=\, \partial_i\bullet A_j-\partial_j\bullet A_i \, \text{ for all }\,i \neq j  \right\}
 \, \subset \, \mathbb{M}_{m\times m}^n\left(\CC(x_1,\ldots,x_n)\right)
\end{align*} 
from an algebro-geometric perspective.
It is a subset of the algebraic variety of $n$-tuples of $m\times m$ matrices with entries in $\CC(x_1,\ldots,x_n)$. It is a diffiety, cut out by first-order algebraic differential equations, namely the integrability conditions. As such, this would no longer be part of the field of algebraic analysis, but rather of differential algebra~\cite{Ritt}. We believe that this algebro-geometric setup is promising to tackle a systematic study of $\varepsilon$-factorized connection matrices. In this case as well as the case of connection matrices with logarithmic differentials as entries---or, more generally whenever $\d A=0$---the $A_i$'s are pairwise commuting. One hence finds oneself within the variety of $n$-tuples of pairwise commuting matrices, which is inherently linked to Quot and Hilbert schemes.  

\bigskip \bigskip
\noindent {\bf Acknowledgments.}
We thank Felix Lotter for helping us with homological arguments to prove \Cref{prop:rankconstant} and Bernd Sturmfels for several inspiring and insightful discussions on Hilbert schemes of points and their relation to the theory of border bases in polynomial~rings. 

\bigskip
\noindent {\bf Funding statement.}
The research of CR and ALS is funded by the European Union (ERC, UNIVERSE PLUS, 101118787). Views and opinions expressed are, however, those of the author(s) only and do not necessarily reflect those of the European Union or the European Research Council Executive Agency. Neither the European Union nor the granting authority can be held responsible for them.

\phantomsection 
\addcontentsline{toc}{section}{References}
\begin{small}

\end{small}

\end{document}